\documentclass[11pt,reqno]{amsart}
\usepackage{amsmath}

\usepackage{amsfonts}
\usepackage{amsthm}

\newtheorem{proposition}{Proposition}

\usepackage{amssymb}
\usepackage[dvips]{graphicx}
\usepackage{color}
\usepackage[pagewise]{lineno}
\usepackage{url}
\theoremstyle{plain}
\newtheorem{athm}{Theorem}

\newtheorem{sublemma}{Sublemma}[section]
\newtheorem{theorem}{Theorem}[section]
\newtheorem{corollary}[theorem]{Corollary}
\newtheorem{lemma}[theorem]{Lemma}

\theoremstyle{definition}

\newtheorem{example}{Example}[section]
\newtheorem{definition}[theorem]{Definition}

\newtheorem{remark}[theorem]{Remark}

\DeclareMathOperator{\esssup}{ess \ sup}

\DeclareMathOperator{\lip}{Lip}

\input{tcilatex}

\begin{document}

\begin{abstract}
    We investigate the statistical stability of a class of dynamical systems semi-conjugate to pre-piecewise \textit{convex or expanding} maps with countably many branches. These systems naturally arise in the study of transformations with unbounded derivatives, discontinuities, or infinite Markov partitions; features that pose significant challenges for stability analysis. Specifically, we consider one-parameter families of transformations $\{F_\delta\}_{\delta \in [0,1)}$ and their corresponding invariant measures $\{\mu_\delta\}$. We provide general conditions ensuring that the unperturbed measure $\mu_0$ is statistically stable, meaning the map $\delta \mapsto \mu_\delta$ is continuous at $\delta = 0$ in the appropriate topology. Furthermore, we establish explicit quantitative estimates for the modulus of continuity of $\mu_\delta$ in terms of the perturbation parameter $\delta$. Our results apply to a broad class of maps, including those semi-conjugate to classical examples such as the Gauss and Lüroth maps.
\end{abstract}

\title[Stability for maps s.c. to pre-piecewise convex or expanding maps]{Statistical stability for systems semi-conjugate to pre-piecewise \textit{convex or expanding} maps with countably many branches}

\author[Rafael Lucena]{Rafael Lucena}

\date{\today }


\address[Rafael Lucena]{Universidade Federal de Alagoas, Instituto de Matemática - UFAL, Av. Lourival Melo Mota, S/N
	Tabuleiro dos Martins, Maceio - AL, 57072-900, Brasil}
\email{rafael.lucena@im.ufal.br}
\urladdr{www.im.ufal.br/professor/rafaellucena}
\maketitle


\section{Introduction}

In many dynamical systems, information about long-term behavior is encoded in invariant probability measures.  
When such a measure varies in a controlled way under perturbations of the system, one gains predictive power: small changes in the dynamics lead to small changes in the statistical description of orbits.  
This phenomenon, known as \emph{statistical stability}, plays a central role in both theoretical investigations and applications, as it connects the structural properties of a system with its measurable outcomes.

In this work, we study statistical stability for a class of dynamical systems that are semi-conjugate to pre-piecewise convex or expanding maps with countably many branches.  
These systems often exhibit complex behaviors such as unbounded derivatives, discontinuities, or infinite Markov partitions, making stability analysis a delicate task.  

Formally, let $\{F_{\delta}\}_{\delta \in [0,1)}$ be a one-parameter family of transformations on a suitable measurable space, with $F = F_0$ as the reference system.  
Assume that for each $\delta$ there exists an $F_{\delta}$-invariant probability measure $\mu_{\delta}$.  
We say that $\mu_0$ is \emph{statistically stable} if $\delta \mapsto \mu_{\delta}$ is continuous at $0$ in a chosen topology.  
Our goal is to establish conditions under which this continuity holds and to provide explicit estimates for its modulus.

To establish our results, we work in a suitable vector space of
signed measures containing the invariant probabilities
$\{\mu_\delta\}_{\delta\in[0,1)}$. For each $\delta$, the transfer
operator associated with $F_\delta$ is the pushforward operator
$F_{\delta*}$, defined by
\[
    (F_{\delta*}\nu)(A)
    := \nu(F_\delta^{-1}(A))
\]
for every finite signed measure $\nu$ and every measurable set $A$
in the phase space. In particular, we write $F_*:=F_{0*}$.
A probability measure $\nu$ is $F_\delta$-invariant if and only if
$F_{\delta*}\nu=\nu$. Under our assumptions, $\mu_\delta$ is the
unique invariant probability measure in the chosen space.
Thus, statistical stability can be studied through the dependence
of these normalized fixed points on the parameter $\delta$.

The transfer operator approach to statistical stability has
a long history, including Keller's work on stochastic
stability \cite{K82}, the spectral stability results of
Baladi and Young \cite{BY93}, and the abstract perturbation
theorem of Keller and Liverani \cite{KL99}.
In particular, Remark~5 of \cite{KL99} discusses a modulus
of order \( \tau|\log\tau| \), where \( \tau \) measures
the operator perturbation, providing an antecedent for
the logarithmic correction appearing in our estimates.
Anisotropic Banach spaces also play a central role in
the spectral study of hyperbolic systems and their SRB
measures, as in Gou\"ezel and Liverani \cite{GoL06}.
The present work studies quantitative stability in the
setting of systems semi-conjugate to pre-piecewise convex
or expanding maps with countably many branches, using
the spaces of measures developed in \cite{L2}.

The main results presented here, theorems \ref{thshgf} and \ref{htyttigu}, are derived by applying previous theorems for transformations semi-conjugate to piecewise convex or expanding systems, as established in \cite{L2}, using an approach developed in this paper for perturbations of operators, inspired by \cite{RRRSTAB}.

Other quantitative results on the stability of different classes of dynamical systems can be found in \cite{GL1} and \cite{GLu}. Recent applications of results of this kind can be found in \cite{AO}, where the authors proved the stability of entropy for piecewise expanding maps.

\subsection*{Organization of the paper}

This paper is organized as follows. In Section~\ref{kjdfkjdsfkj}, we detail the hypotheses governing the dynamical systems under study and provide examples illustrating the role of each assumption. Section~\ref{uyrtghfdfdcsdf,mbnv} summarizes key results regarding the spectral properties of transfer operators acting on anisotropic Banach spaces, following the framework of \cite{L2}, which serves as the foundation for our analysis. Section~\ref{76235ghjdf} defines the class of perturbations considered, establishes their fundamental properties, and presents the proofs of our main results, Theorems~A and~B. Finally, Section~\ref{jhgfds} is devoted to the proof of Theorem~\ref{tc}, which ensures that the class of perturbations investigated arises naturally in the context of important systems semi-conjugated to Gauss and Luroth maps.

\textbf{Acknowledgments} The author is supported by CNPq (Brazil) Grants 446515/2024-8, CNPq (Brazil) Grants 420353/2025-9.

I am grateful to my uncle Paulo Nóbrega and my aunt Marli Virginia
for generously letting me use their apartment overlooking the sea.
It was in this peaceful setting that I completed much of the work
presented here.

\section{Preliminaries and Examples}\label{kjdfkjdsfkj}

Throughout the paper, we adopt the notation \( I = [0,1] \), with \( \mathcal{B} \) denoting the Borel \( \sigma \)-algebra on \( I \), and \( m \) the Lebesgue measure normalized on \( \mathcal{B} \). The standard Euclidean metric on \( I \) is denoted by \( d_1 \), and \( K \) stands for a compact metric space equipped with a metric \( d_2 \). Without loss of generality, we assume \( \mathrm{diam}(K) = 1 \) to simplify expressions. We consider the product space \( \Sigma := I \times K \), endowed with the metric \( d_1 + d_2 \), making \( \Sigma \) itself a metric space.

In this section, we introduce the assumptions underlying the dynamical system considered throughout the paper: 
\[
F: \Sigma \to \Sigma, \quad F(x,y) = (f(x), G(x,y)),
\]
where \( f: I \to I \) is a measurable transformation on the interval and \( G: \Sigma \to K \) is a measurable function that governs the dynamics on the fiber. 

The canonical projection $\pi_1:\Sigma\to I$,
defined by $\pi_1(x,y)=x$, satisfies
\[
    \pi_1\circ F=f\circ\pi_1.
\]
Thus, $F$ is semiconjugate to its base map $f$.

\subsection{Pre-piecewise convex or expanding maps with countably many branches}

Let $I=[0,1]$ and let $\mathcal {P}=\{I_i=(a_i,b_i)\}_{i=1}^\infty$ be a family of open disjoint subintervals of $I$ such that $m\left(I\setminus\bigcup_{i\ge 1}I_i\right)=0$.

\begin{definition}\label{def1}
Following \cite[Definition~2.1]{RG2}, we say that a map
\( f:\bigcup_{i\geq1} I_i \to I \) is a
\emph{piecewise convex map with countably many branches}
on the partition \( \mathcal{P} \) if the following
conditions are satisfied:
	
	\begin{enumerate}
		\item For each \( i \geq 1 \), the restriction \( f_i := f|_{I_i} \) is an increasing, convex, and differentiable function such that
		\[
		\lim_{x \to a_i^+} f_i(x) = 0.
		\]
		Define \( f_i(a_i) := 0 \) and \( f_i(b_i) := \lim_{x \to b_i^-} f_i(x) \). The value \( f_i'(a_i) \) is also defined by continuity from the right.
		
		\item The derivatives at the endpoints satisfy the summability condition:
		\[
		\sum_{i \geq 1} \frac{1}{f_i'(a_i)} < +\infty.
		\]
		
		\item If \( x = 0 \) is not a limit point of the partition points \( \{a_i\}_{i \geq 1} \), then
		\[
		f'(0) = \frac{1}{\beta} > 1,
		\]
		for some constant \( 0 < \beta < 1 \).
	\end{enumerate}
\end{definition}

In accordance with \cite[Definition~2.1]{RG2}, convexity
is understood in the non-strict sense. Since each branch
\( f_i \) is differentiable, this is equivalent to
\( f_i' \) being non-decreasing on \( I_i \).

\begin{definition}\label{def2}
	We say that a map \( f: \bigcup_{i \geq 1} I_i \to I \) is a \emph{piecewise expanding map with countably many branches} on the partition \( \mathcal{P} \) if, for each \( i \geq 1 \), the restriction \( f|_{I_i} \) extends to a homeomorphism \( f_i: [a_i, b_i] \to f_i([a_i, b_i]) \).
	
	Moreover, define the function \( g: I \to \mathbb{R} \) by
	\[
	g(x) = 
	\begin{cases}
		\dfrac{1}{|f_i'(x)|}, & \text{if } x \in I_i \text{ for some } i \geq 1, \\
		0, & \text{otherwise}.
	\end{cases}
	\]
	We require that \( \sup_{x \in I} |g(x)| \leq \beta < 1 \).

    We also require $V_I(g)<\infty$, where
$V_I(g)$ denotes the pointwise variation of
the representative defined above, including
its boundary contributions.
\end{definition}

The proof of the following theorem can be found in \cite{RG2}.

\begin{theorem}\label{raj} If $\tau$ satisfies Definition \ref{def1}, then some iterate  $\tau^n$ satisfies Definition \ref{def2} for all sufficiently large $n\in \mathbb{N}$.
\end{theorem}

\subsubsection{Pre monotonic or expanding systems and bounded variation functions}

A measurable map \( f: \bigcup_{i \geq 1} I_i \to I \) is said to be \emph{non-singular} if the pushforward of the Lebesgue measure \( m \) by \( f \), denoted \( f_* m \), satisfies \( f_* m \ll m \).

\begin{definition}\label{bv}
Given a function \( h:I\to\mathbb{R} \) and a subset
\( J\subset I \), define the variation of \( h \) on \( J \) by
\[
    V_J(h):=
    \sup\left\{
        \sum_{i=1}^{k}|h(x_i)-h(x_{i-1})|:
        k\geq1,\;
        x_0,\ldots,x_k\in J,\;
        x_0\leq\cdots\leq x_k
    \right\}.
\]
For \( J=\varnothing \), we set \( V_J(h)=0 \).

Let \( L^1_m \) denote the space of equivalence classes
of real-valued, \( m \)-integrable functions on \( I \).
For \( h\in L^1_m \), define
\[
    V(h):=
    \inf\left\{
        V_I(h^*):
        h^*:I\to\mathbb{R}
        \text{ is a representative of } h
    \right\}.
\]
Thus, the infimum is taken over all representatives
of the equivalence class \( h \). We set
\[
    BV_m:=\{h\in L^1_m:V(h)<+\infty\},
\]
endowed with the norm
\[
    |h|_v:=\int_I |h|\,dm+V(h).
\]
\end{definition}

Given a non-singular map \( f: \bigcup_{i \geq 1} I_i \to I \), we define its Frobenius–Perron operator \( \func{P}_f: L^1_m \to L^1_m \) via the duality relation
\[
\int h_1 \, \func{P}_f(h_2) \, dm = \int (h_1 \circ f) \, h_2 \, dm,
\]
for all \( h_1 \in L^\infty_m \) and \( h_2 \in L^1_m \), where \( L^p_m \) denotes the usual Lebesgue space of \( p \)-integrable functions with respect to \( m \) (for \( p = 1, \infty \)).

\begin{definition}\label{def3}	
	We say that a transformation \( f: \bigcup_{i \geq 1} I_i \to I \) is a \emph{pre-piecewise convex or expanding map with countably many branches} on the partition \( \mathcal{P} \) if the following conditions are satisfied:
	\begin{enumerate}
		\item There exists an integer \( n_0 \in \mathbb{N} \) such that the iterate \( f^{n_0} \) satisfies either Definition~\ref{def1} or Definition~\ref{def2};
		
		\item The map \( f: I \to I \) is non-singular with respect to the Lebesgue measure \( m \);
		
		\item The associated Frobenius–Perron operator \( \func{P}_f: BV_m \to BV_m \) is bounded;
		
		\item There exists a unique $f$-invariant
probability measure absolutely continuous
with respect to $m$. This measure has the
form $m_1=h_1m$, where $h_1\in BV_m$,
and is mixing.
	\end{enumerate}
\end{definition}

\begin{remark}\label{rem1}
If $f$ satisfies Definition~\ref{def1}, then
it satisfies Definition~\ref{def3} with
$n_0=1$, by the results of G\'ora and
Rajput~\cite{RG2}.

If $f$ satisfies Definition~\ref{def2},
Rychlik's results~\cite{RLK} give boundedness
and quasi-compactness of its Frobenius--Perron
operator on $BV_m$. If, in addition, its
absolutely continuous invariant probability
measure is unique and mixing, then $f$
satisfies Definition~\ref{def3} with $n_0=1$.
\end{remark}

The following spectral decomposition follows
from the quasi-compactness results of
G\'ora and Rajput~\cite{RG2} and
Rychlik~\cite{RLK}, applied to a suitable
iterate, together with the boundedness of
$P_f$ on $BV_m$ and the uniqueness and
mixing assumptions in Definition~\ref{def3}.

	\begin{theorem}\label{gapf}
	Suppose that \( f: I \to I \) satisfies Definition \ref{def3}. Then the operator \( \func{P}_f \) has a spectral gap on the space \( BV_m \).
		
		More precisely, we have\\
		(1) $\func{P}_f: BV_m \rightarrow BV_m$ has $1$ as the only eigenvalue of modulus $1$.\\
		(2) $E_1:=\{h\in L_m^1 : \func{P}_f h=h\} \subseteq BV_m$ and $E_1$ is one dimensional.\\
		(3) $\displaystyle \func{P}_f= \Psi +Q$, where $\Psi$ represents the projection on eigenspace  $E_1$, $|\Psi|_1\leq 1$ and $Q$ is a linear operator on $L_m^1$ with $Q(BV_m) \subseteq BV_m$, $\displaystyle \sup_{n\in \mathbb{N}} |Q^n|_1 < \infty$ and $Q \cdot \Psi = \Psi \cdot Q  = 0$.\\
		(4) $Q(BV_m)\subset BV_m$ and, considered as a linear operator on $(BV_m,|\cdot|_v)$, $Q$ satisfies $|Q^n|_v \leq H \cdot q^n$  $(n\geq 1)$ for some constants $H>0$ and $0<q<1$. 
\end{theorem}

\subsection{Assumptions on $G$}

We assume that the map \( G: \Sigma \to K \) satisfies the following condition:

\begin{enumerate}
	\item[(H1)] \( G \) is uniformly contracting along \( m \)-almost every vertical fiber \( \gamma_x := \{x\} \times K \); that is, there exists a constant \( 0 \leq \alpha < 1 \) such that for \( m \)-almost every \( x \in I \),
	\begin{equation}\label{contracting1}
		d_2(G(x, z_1), G(x, z_2)) \leq \alpha \, d_2(z_1, z_2), \quad \text{for all } z_1, z_2 \in K.
	\end{equation}
\end{enumerate}

We denote by \( \mathcal{F}^s \) the family of vertical fibers:
\[
\mathcal{F}^s := \left\{ \gamma_x := \{x\} \times K \; ; \; x \in I \right\}.
\]
When the context is clear, we will refer to elements of \( \mathcal{F}^s \) simply by \( \gamma \) instead of \( \gamma_x \).

\begin{remark}
	Elements of \( \mathcal{F}^s \), that is, the fibers \( \gamma_x = \{x\} \times K \) for \( x \in I \), are naturally associated with their base point \( x \). For convenience, we will sometimes use the same symbol to refer to both the fiber \( \gamma \) (suppressing the subscript \( x \)) and the point \( x \), whenever the context allows. For example, if \( \phi: I \to \mathbb{R} \) is a function, the expressions \( \phi(x) \) and \( \phi(\gamma) \) will be used interchangeably, identifying \( \gamma \) with its base point \( x \).
\end{remark}

\begin{enumerate}
	\item[(H2)] Let \( \{I_s\}_{s \in \mathbb{N}} \) be the partition of \( I \) as given by Definitions~\ref{def1} or~\ref{def2}. Assume that, for each \( s \in \mathbb{N} \), the map \( G \) satisfies the uniform Lipschitz condition:
	\begin{equation}\label{oityy}
		|G_s|_{\mathrm{Lip}} := \sup_{y \in K} \sup_{x_1, x_2 \in I_s} \frac{d_2(G(x_1, y), G(x_2, y))}{d_1(x_1, x_2)} < \infty.
	\end{equation}
\end{enumerate}

We define the global Lipschitz constant of \( G \) along the horizontal direction as
\begin{equation}\label{jdhfjdh}
	|G|_{\mathrm{Lip}} := \sup_{s \in \mathbb{N}} |G_s|_{\mathrm{Lip}} < \infty.
\end{equation}

\begin{remark}
	Condition (H2) allows for the possibility that \( G \) is discontinuous along the sets \( \partial I_i \times K \), for all \( i \in \mathbb{N} \), where \( \partial I_i \) denotes the boundary of the interval \( I_i \).
\end{remark}

\begin{remark}\label{uyrytuert}
	In certain cases, the map \( G \) may exhibit discontinuities along any countable collection of vertical fibers of the form \( \{x\} \times K \), with \( x \in I \). For further discussion, we refer the reader to Remark~2.8 in~\cite{L2}.
\end{remark}

\begin{enumerate}
	\item[(H3)] There exists an integer \( k \in \mathbb{N} \) such that the \( k \)-th iterate \( F^k = (f^k, G_k) \) satisfies the inequality
	\[
	\alpha_4 := \alpha^k \cdot \esssup \left( \frac{1}{|(f^k)'|} \right) < 1,
	\]
	where the essential supremum is taken with respect to the Lebesgue measure \( m \). Furthermore, the fiber map \( G_k \) satisfies the fiberwise contraction condition \textnormal{(H1)} with contraction rate \( \alpha^k \), and the branchwise horizontal Lipschitz condition
\textnormal{(H2)}.
\end{enumerate}

\begin{definition}\label{closed}
	We say that the branchwise horizontal Lipschitz condition
\textnormal{(H2)} is \emph{closed under iteration} if,
for every \( n\geq1 \), the fiber component \( G_n \),
defined by \( F^n=(f^n,G_n) \), also satisfies
\textnormal{(H2)} with respect to the branch partition
of \( f^n \).
\end{definition}

\begin{remark}
	If the branchwise horizontal Lipschitz condition
\textnormal{(H2)} is closed under iteration in the sense
of Definition~\ref{closed}, then \textnormal{(H3)} holds.
\end{remark}

\subsection{Examples}\label{dkjfhksjdhfksdf}

\subsubsection{Illustrative Examples of the Map $f$}

\begin{example}[Piecewise Linear Expanding Map]\label{slopes1}
	Let \( I = [0,1] \) and denote by \( m \) the Lebesgue measure on \( I \). Consider a countable collection of pairwise disjoint open intervals \( (I_i)_{i \in \mathbb{N}} \) such that
	\[
	m\left(I \setminus \bigcup_{i \geq 1} I_i\right) = 0.
	\]
	Define a map \( f: \bigcup_{i=1}^\infty I_i \to [0,1] \) such that each restriction \( f_i := f|_{I_i} \) is linear with constant slope \( k_i \). Assume further that
	\[
	\inf_{i \in \mathbb{N}} k_i > 1 \quad \text{and} \quad \sum_{i=1}^\infty \frac{1}{k_i} < +\infty.
	\]
	Under these conditions, the map \( f \) satisfies Definition~\ref{def2}. For additional discussion and technical details, see~\cite{RLK}.
\end{example}

\begin{example}[Piecewise Linear Map with Mixed Slopes]\label{slopes2}
	Let \( I = [0,1] \), and let \( m \) denote the Lebesgue measure on \( I \). Consider a countable collection of pairwise disjoint open intervals \( (I_i)_{i \in \mathbb{N}} \) such that
	\[
	m\left(I \setminus \bigcup_{i \geq 1} I_i\right) = 0.
	\]
	Define a map \( f: \bigcup_{i=1}^\infty I_i \to [0,1] \) such that each branch \( f_i := f|_{I_i} \) is linear with slope \( k_i \), and assume that condition (1) of Definition~\ref{def1} holds.
	
	Suppose that \( k_i \in (0,1) \) for only finitely many indices \( i \geq 2 \), and denote this finite set by \( \mathbb{N}_1 \). Assume further that \( \inf_{i \in \mathbb{N}_1^c} k_i > 1 \) and
	\[
	\sum_{i=1}^\infty \frac{1}{k_i} < + \infty.
	\]
	Under these assumptions, the map \( f \) satisfies Definition~\ref{def1}, but does not satisfy Definition~\ref{def2} at time one. In particular, \( f'(x) < 1 \) for all \( x \in \bigcup_{i \in \mathbb{N}_1} I_i \).
\end{example}

\begin{remark}
The following example presents a map that fails to satisfy Definitions~\ref{def1} and~\ref{def2} at time one, but whose second iterate does satisfy both. This provides a strict and nontrivial instance of a pre-piecewise convex or expanding map, as defined in Definition~\ref{def3}, with \( n_0 > 1 \). 
\end{remark}

\begin{example}[Gauss Map]\label{gauss}
	Let \( \mathcal{P} = (I_i)_{i=1}^\infty \) be the partition of the interval \( [0,1] \) given by
	\[
	I_i = \left( \frac{1}{i+1}, \frac{1}{i} \right), \quad \text{for all } i \in \mathbb{N}.
	\]
	Define the map \( f: \bigcup_{i=1}^\infty I_i \to [0,1] \) by
	\[
	f(x) = \frac{1}{x} - i, \quad \text{for all } x \in I_i.
	\]
	The second iterate $f^2$ is uniformly expanding,
and its reciprocal derivative, extended by
zero outside the branch interiors, has finite
pointwise variation. Thus $f^2$ satisfies
Definition~\ref{def2}.
	
	Furthermore, $f$ has a unique absolutely
continuous invariant probability measure
$m_1$, which is mixing, and its
Frobenius--Perron operator is bounded on $BV_m$.
Hence $f$ satisfies Definition~\ref{def3}
with $n_0=2$. The density of $m_1$ is
	\[
	h_1(x) = \frac{1}{(1 + x)\log 2}.
	\]
	For additional details, we refer the reader to~\cite{Kva},~\cite{WP}, and~\cite{WP2}.
\end{example}

\begin{example}[$\mathcal{P}$-Lüroth Maps]\label{luroth}
	Let \( \mathcal{P} = (I_i)_{i=1}^\infty \) be a countable partition of the interval \( [0,1] \), where each \( I_i \) is a non-empty, left-open and right-closed interval. Assume that the intervals \( I_i \) are ordered from right to left, beginning with \( I_1 \), and that they accumulate only at \( 0 \).
	
	Let \( a_i := m(I_i) \) denote the Lebesgue measure of \( I_i \), and define the tail sum \( t_i := \sum_{k=i}^\infty a_k \), which corresponds to the Lebesgue measure of the \( i \)-th tail of the partition.
	
	The \(\mathcal{P}\)-Lüroth map \( f_{\mathcal{P}}: [0,1] \to [0,1] \) is defined by
	\[
	f_{\mathcal{P}}(x) = 
	\begin{cases}
		\dfrac{t_i - x}{a_i}, & \text{if } x \in I_i \text{ for some } i \geq 1, \\
		0, & \text{otherwise}.
	\end{cases}
	\]
	It can be shown that \( f_{\mathcal{P}} \) satisfies Definition~\ref{def2}. For further information and related results, see~\cite{MSO}.
\end{example}

\subsubsection{Illustrative Examples of the Fiber Map $G$}

\begin{example}[Discontinuous Maps with Constant Fiber Contractions]\label{h3}
	Let \( F = (f, G) \) be a measurable map, where \( f \) satisfies Definition~\ref{def1} or Definition~\ref{def2} for some iterate \( f^n \). Consider a sequence of real numbers \( \{\alpha_i\}_{i=1}^\infty \) such that
	\[
	0 \leq \alpha_i < \alpha_{i+1} \leq \alpha < 1, \quad \text{for all } i \geq 1.
	\]
	Define \( G: [0,1] \times [0,1] \to [0,1] \) by
	\[
	G(x, y) = \alpha_i y, \quad \text{for all } x \in I_i, \ y \in [0,1], \text{ and each } i \geq 1.
	\]
	Clearly, \( G \) is discontinuous along the sets \( \partial I_i \times [0,1] \) for all \( i \in \mathbb{N} \).
	
	Since \( G(x,y)=\alpha_i y \) is independent of \( x \)
on each strip \( I_i\times[0,1] \), it satisfies the
branchwise horizontal Lipschitz condition
\textnormal{(H2)}, with \( |G|_{\mathrm{lip}}=0 \)
in \eqref{jdhfjdh}. Moreover, \( \alpha_i\leq\alpha<1 \)
ensures the fiberwise contraction condition
\textnormal{(H1)}.
\end{example}

\begin{example}[Discontinuous Maps with Lipschitz Fiber Coefficients]
	Let \( F = (f, G) \) be a measurable map, where \( f \) satisfies Definition~\ref{def1} or Definition~\ref{def2} for some iterate \( f^n \). Denote the atoms of the partition \( \mathcal{P} \) by \( I_i := (a_i, b_i) \), for all \( i \geq 1 \), as in Definitions~\ref{def1} and~\ref{def2}.
	
	Fix a real number \( 0 \leq \alpha < 1 \), and let \( \{h_i\}_{i=1}^\infty \) be a sequence of real-valued Lipschitz functions such that \( h_i : I_i \to [0,1] \), \( 0 \leq h_i(x) \leq \alpha < 1 \) for all \( x \in I_i \), and
	\[
	\sup_{i \geq 1} L(h_i) < \infty,
	\]
	where \( L(h_i) \) denotes the Lipschitz constant of \( h_i \). Assume also that the functions are not continuous across the partition boundaries, i.e., \( h_i(b_i) \neq h_{i+1}(a_{i+1}) \) for all \( i \geq 1 \).
	
	Define \( G : [0,1] \times [0,1] \to [0,1] \) by
	\[
	G(x, y) = h_i(x) y, \quad \text{for all } x \in I_i \text{ and } y \in [0,1].
	\]
	
	Under these assumptions, \( G \) is discontinuous along the sets \( \partial I_i \times [0,1] \) for all \( i \in \mathbb{N} \). 
    
   The uniform bound on the Lipschitz constants of the
functions \( h_i \) ensures the branchwise horizontal
Lipschitz condition \textnormal{(H2)}, with
\[
    |G|_{\mathrm{lip}}\leq\sup_{i\geq1}L(h_i)<+\infty
\]
in \eqref{jdhfjdh}. Moreover,
\( 0\leq h_i(x)\leq\alpha<1 \) ensures the fiberwise
contraction condition \textnormal{(H1)}.
	
	\begin{remark}
		This example illustrates how fiberwise contraction can coexist with horizontal discontinuities. Even though the map \( G \) is not continuous along vertical lines crossing partition boundaries, its structure allows for uniform contraction along the fibers and controlled horizontal regularity, ensuring that assumptions (H1) and (H2) remain valid.
	\end{remark}
\end{example}

\begin{example}[Closure of (H2) under iteration]
	Consider a map \( F = (f, G) \), where \( f \) is as in Example~\ref{slopes2} and \( G \) is defined as in Example~\ref{h3}. The branchwise horizontal Lipschitz condition
\textnormal{(H2)} is closed under iteration in the
sense of Definition~\ref{closed}: on each branch
of \( f^n \), the fiber component \( G_n(x,y) \)
is independent of \( x \). Consequently, the
iterate condition \textnormal{(H3)} also holds.
	
	Moreover, the sequence \( \{\alpha_i\}_i \) can be chosen so that the inequality \( \alpha \cdot \esssup g < 1 \) fails to hold at time one, despite holding for some iterate. This illustrates that the contraction required in (H3) can emerge only after iteration. In particular, this construction aligns with the scenario described in Remark~\ref{uyrytuert}, where discontinuities may occur along countably many vertical fibers.
\end{example}

\section{Previous Results on the Transfer Operator of \(F\)}\label{uyrtghfdfdcsdf,mbnv}
	
	We recall in this section some known results on the transfer operator of \( F \), all originating from \cite{L2}. 
	
	Consider a signed measure \( \mu \) defined on the Borel \( \sigma \)-algebra of \( \Sigma := I \times K \). The transfer operator associated with \( F \) is the pushforward map \( \mu \mapsto \func{F}_* \mu \), defined by
	\[
	\func{F}_* \mu (A) = \mu(F^{-1}(A))
	\]
	for every Borel set \( A \subset \Sigma \). 
	
\subsection{Anisotropic Spaces of Signed Measures}
	
\subsubsection{The Spaces $\mathcal{L}^{1}$ and $S^{1}$}\label{jdfjdhkjf}

Let \( \mathcal{SB}(\Sigma) \) denote the space of Borel signed measures on \( \Sigma := I \times K \). For \( \mu \in \mathcal{SB}(\Sigma) \), let \( \mu^+ \) and \( \mu^- \) be the positive and negative parts of its Jordan decomposition, so that \( \mu = \mu^+ - \mu^- \). Let \( \pi_1 : \Sigma \to I \) be the projection \( \pi_1(x,y) = x \), and denote by \( \pi_{1*} \) the corresponding pushforward operator.

Define the class \( \mathcal{AB} \subset \mathcal{SB}(\Sigma) \) as the set of signed measures \( \mu \) whose marginals \( \pi_{1*} \mu^+ \) and \( \pi_{1*} \mu^- \) are absolutely continuous with respect to the Lebesgue measure \( m \):
\[
\mathcal{AB} := \left\{ \mu \in \mathcal{SB}(\Sigma) : \pi_{1*} \mu^+ \ll m \ \text{and} \ \pi_{1*} \mu^- \ll m \right\}.
\]

If \( \mu \in \mathcal{AB} \) is a probability measure, then by Rokhlin's disintegration theorem (Theorem 3.1 of \cite{L2}), there exists a disintegration \( (\{ \mu_\gamma \}_\gamma, \mu_1) \) along the stable partition \( \mathcal{F}^s \), where \( \mu_1 = \pi_{1*} \mu = \phi_1 m \) and \( \{ \mu_\gamma \}_\gamma \) is a measurable family of probability measures supported on the fibers \( \gamma = \{x\} \times K \).

Moreover, since \( \mu_1 \ll m \), the density \( \phi_1 : I \to \mathbb{R} \) is non-negative, defined almost everywhere, and satisfies \( \|\phi_1\|_{L^1(m)} = 1 \). For a non-normalized positive measure \( \mu \in \mathcal{AB} \), a similar disintegration holds, with the same properties, except that \( \|\phi_1\|_{L^1(m)} = \mu(\Sigma) \).

The disintegration \( (\{ \mu_\gamma \}_\gamma, \mu_1) \) satisfies:

\begin{enumerate}
	\item[(a)] \( \mu_\gamma(\gamma) = 1 \) for \( \mu_1 \)-almost every \( \gamma \in \mathcal{F}^s \);
	\item[(b)] For every measurable set \( E \subset \Sigma \), the map \( \gamma \mapsto \mu_\gamma(E) \) is measurable;
	\item[(c)] For every measurable set \( E \subset \Sigma \), we have
	\[
	\mu(E) = \int \mu_\gamma(E) \, d\mu_1(\gamma) = \int \mu_\gamma(E) \phi_1(\gamma) \, dm(\gamma);
	\]
	\item[(d)] If the \( \sigma \)-algebra on \( \Sigma \) has a countable generator, then the disintegration is unique: any other disintegration \( (\{ \mu_\gamma' \}_\gamma, \mu_1) \) satisfies \( \mu_\gamma = \mu_\gamma' \) for \( \mu_1 \)-almost every \( \gamma \).
\end{enumerate}

\begin{definition}\label{restrictionmeasure}
	Let \( \pi_2 : \Sigma \to K \) be the projection \( \pi_2(x,y) = y \), and let \( \pi_{\gamma,2} \) denote its restriction to a fiber \( \gamma = \{x\} \times K \). Given a measure \( \mu \in \mathcal{AB} \) with disintegration \( (\{ \mu_\gamma \}_\gamma, \mu_1 = \phi_1 m) \), we define the \textbf{restriction of \( \mu \) to \( \gamma \)} as the positive measure on \( K \) given by
	\[
	\mu|_\gamma(A) := \pi_{\gamma,2*}(\phi_1(\gamma)\mu_\gamma)(A), \quad \text{for all measurable } A \subset K.
	\]
	
	If \( \mu \in \mathcal{AB} \) is a signed measure with Jordan decomposition \( \mu = \mu^+ - \mu^- \), then its restriction to \( \gamma \) is defined as
	\[
	\mu|_\gamma := \mu^+|_\gamma - \mu^-|_\gamma.
	\]
\end{definition}

\begin{remark}\label{ghtyhh}
	As shown in Appendix 2 of \cite{GLu}, the definition of \( \mu|_\gamma \) is independent of the chosen decomposition. That is, if \( \mu = \nu_1 - \nu_2 \) for any pair of positive measures \( \nu_1, \nu_2 \), then \( \mu|_\gamma = \nu_1|_\gamma - \nu_2|_\gamma \) for \( \mu_1 \)-almost every \( \gamma \in I \).
\end{remark}

	Let \( (X, d) \) be a compact metric space and let \( h : X \to \mathbb{R} \) be a Lipschitz function. Its optimal Lipschitz constant is given by
	\begin{equation}\label{lipsc}
		L(h) := \sup_{x \neq y} \frac{|h(x) - h(y)|}{d(x, y)}.
	\end{equation}
	
	\begin{definition}\label{wasserstein}
		Given signed measures \( \mu \) and \( \nu \) on \( X \), we define the \textbf{Wasserstein–Kantorovich} type distance as
		\[
		W_1(\mu, \nu) := \sup_{\substack{L(h) \leq 1 \\ |h|_\infty \leq 1}} \left| \int h \, d\mu - \int h \, d\nu \right|.
		\]
	\end{definition}
	
	We also write
	\begin{equation}\label{WW}
		\|\mu\|_W := W_1(0, \mu),
	\end{equation}
	noting that \( \|\cdot\|_W \) defines a norm on the space of signed measures on \( X \), equivalent to the dual norm of the Lipschitz space.
	
	\begin{definition}\label{sdfsdfsdasd}
		Let \( \mathcal{L}^1 \subset \mathcal{AB}(\Sigma) \) be the space of signed measures whose vertical disintegration satisfies
		\[
		\mathcal{L}^1 := \left\{ \mu \in \mathcal{AB} : \int W_1(\mu^+|_\gamma, \mu^-|_\gamma) \, dm < \infty \right\}.
		\]
		Define the functional \( \|\cdot\|_1 : \mathcal{L}^1 \to \mathbb{R} \) by
		\[
		\|\mu\|_1 := \int \|\mu|_\gamma\|_W \, dm(\gamma),
		\]where $\|\mu|_\gamma\|_W = W_1(\mu^+|_\gamma, \mu^-|_\gamma)$.
		
		Now let
		\[
		S^1 := \left\{ \mu \in \mathcal{L}^1 : \phi_1 \in BV_m \right\},
		\]
		where \( \phi_1 \) is the marginal density associated to \( \mu \), and define the norm
		\[
		\|\mu\|_{S^1} := |\phi_1|_v + \|\mu\|_1.
		\]
	\end{definition}
	
	Both \( (\mathcal{L}^1, \|\cdot\|_1) \) and \( (S^1, \|\cdot\|_{S^1}) \) are normed vector spaces. Their properties are elementary to verify; see, for instance, analogous constructions in \cite{L}.

\subsubsection{The Space of Measures of Bounded Variation}\label{xbcvhgsafd}

Let \( \Sigma := I \times K \). A signed measure \( \mu \) on \( \Sigma \) admits a disintegration along the stable foliation \( \mathcal{F}^s \), yielding a family \( \{ \mu|_\gamma \}_{\gamma \in I} \) of signed measures on \( K \). Since \( \mathcal{F}^s \) is parametrized by \( I \), this disintegration defines (almost everywhere) a measurable map \( \Gamma_\mu : I \to \mathcal{SM}(K) \), where \( \mathcal{SM}(K) \) denotes the space of signed Borel measures on \( K \), endowed with the Wasserstein–Kantorovich distance (see Definition~\ref{wasserstein}). Explicitly,
\[
\Gamma_\mu(\gamma) := \mu|_\gamma,
\]
with \( \mu|_\gamma := \pi_{\gamma, 2*} \left( \phi_1(\gamma) \mu_\gamma \right) \), where \( (\{\mu_\gamma\}, \phi_1) \) is a disintegration of \( \mu \).

Since disintegrations are only defined \( m \)-almost everywhere and are not unique, we define:

\begin{definition}
	Let \( \mu \in \mathcal{AB} \), and let \( \omega = (\{\mu_\gamma\}, \phi_1) \) be any disintegration of \( \mu \). Define the associated path \( \Gamma^\omega_\mu : I_\omega \to \mathcal{SM}(K) \), where \( I_\omega \subset I \) is a full \( m \)-measure subset, by
	\[
	\Gamma^\omega_\mu(\gamma) := \mu|_\gamma = \pi_{\gamma, 2*} \left( \phi_1(\gamma) \mu_\gamma \right).
	\]
	The equivalence class of all such paths associated to \( \mu \) is denoted by
	\[
	\Gamma_\mu := \{ \Gamma^\omega_\mu \}_\omega.
	\]
\end{definition}

\begin{definition}
	Given a path \( \Gamma^\omega_\mu \in \Gamma_\mu \), define its \textbf{variation over \( I_\omega \)} by
	\begin{equation}\label{Lips1}
    V_{I_\omega}(\Gamma^\omega_\mu)
    :=
    \sup_{\substack{
        k\geq1,\;
        x_0<\cdots<x_k\\
        x_0,\ldots,x_k\in I_\omega
    }}
    \sum_{i=1}^{k}
    \left\|
        \Gamma^\omega_\mu(x_i)
        -
        \Gamma^\omega_\mu(x_{i-1})
    \right\|_W.
\end{equation}
	The \textbf{variation of the measure \( \mu \)} is defined as
	\begin{equation}\label{Lips2}
		V_I(\mu) := \inf_{\Gamma^\omega_\mu \in \Gamma_\mu} V_{I_\omega}(\Gamma^\omega_\mu).
	\end{equation}
\end{definition}

\begin{remark}
	Given \( \eta \subset I_\omega \), we denote the variation of \( \Gamma^\omega_\mu \) over \( \eta \) by \( V_\eta(\Gamma^\omega_\mu) \). If \( \eta \subset I \) is an interval, we define
	\[
	V_\eta(\mu) := \inf_{\Gamma^\omega_\mu \in \Gamma_\mu} V_{\eta \cap I_\omega}(\Gamma^\omega_\mu).
	\]
\end{remark}

\begin{remark}
	When no confusion arises, we write \( \mu|_\gamma \) in place of \( \Gamma^\omega_\mu(\gamma) \).
\end{remark}

\begin{definition}
	The space of measures of \textbf{bounded variation} is defined as
	\[
	\mathcal{BV}_m := \left\{ \mu \in \mathcal{AB} : V_I(\mu) < +\infty \right\}.
	\]
\end{definition}
\begin{remark}
	Other sort of regularities of disintegration and their respective applications can be seen in \cite{RR}, \cite{RRR}, \cite{GL1}, and \cite{DR}. 
\end{remark}
	
\subsection{Basic Properties of the Pushforward Operator $\func{F}_*$}\label{invt}

In this section, we present some properties of the linear operator $\func{F}_*$, which will be used to establish the main theorems of this paper.

All results in Sections~\ref{12hg} and~\ref{bdhfdf} were proved in~\cite{L2}. We refer the reader to Corollary~2.19, Proposition~2.20, and Remark~2.21 therein, which explain why the results of~\cite{L2} apply to systems $f$ satisfying Definition~\ref{def3}.

	\subsubsection{Preliminary properties of $\func{F}_*$}\label{12hg}
	
	\begin{proposition}
		\label{niceformulaab}	Let \( F:\Sigma\to\Sigma \), \( F=(f,G) \), be a
transformation whose base map \( f \) satisfies
Definition~\ref{def3} and whose fiber map \( G \)
satisfies the uniform fiberwise contraction condition
\textnormal{(H1)} stated in Section~2.2. Let $\gamma \in \mathcal{F}^{s}$ be a stable leaf. Define the map $F_{\gamma }:K\longrightarrow K$ by 
		\begin{equation}\label{ritiruwt}
			F_{\gamma }=\pi _{2}\circ F|_{\gamma }\circ \pi _{\gamma ,2}^{-1}.
		\end{equation}%
		Then, for each $\mu \in \mathcal{L}^{1}$ and for almost all $\gamma \in
		I$ it holds 
		\begin{equation}
			(\func{F}_{\ast }\mu )|_{\gamma }=\sum_{i=1}^{+\infty}{\func{F}%
				_{\gamma _i \ast }\mu |_{\gamma _i }g _i(\gamma _i)\chi _{f_{i}(I_{i})}(\gamma )}\ \ m%
			\mathnormal{-a.e.}\ \ \gamma \in I  \label{niceformulaa}
		\end{equation}%
		where $\func{F}_{\gamma_i \ast }$ is the pushforward map
		associated to $\func{F}_{\gamma_i}$, $\gamma _i = f_{i}^{-1}(\gamma )$ when $\gamma \in f_i (I_i)$ and $g_i(\gamma)= \dfrac{1}{|f_i^{'}(\gamma)|}$, where $f_i = f|_{I_i}$.
	\end{proposition}
	
	\begin{remark}\label{çlkhgh}
		In Equation~\eqref{niceformulaa}, each term of the form 
		\[
		\func{F}_{\gamma_i *} \left( \mu|_{\gamma_i} \right) g_i(\gamma_i) \chi_{f_i(I_i)}(\gamma)
		\]
		vanishes outside \( f_i(I_i) \). Therefore, we adopt the following more compact expression for the restriction of \( \func{F}_* \mu \) to the leaf \( \gamma \):
		\begin{equation}\label{yrtyer}
			(\func{F}_* \mu)|_\gamma = \sum_{i=1}^{+\infty} \func{F}_{\gamma_i *} \left( \mu|_{\gamma_i} \right) g_i(\gamma_i), \quad m\text{-a.e. } \gamma \in I,
		\end{equation}
		where each \( \gamma_i \) satisfies \( f(\gamma_i) = \gamma \).
	\end{remark}

	\begin{lemma}
		\label{niceformulaac}	Let \( F:\Sigma\to\Sigma \), \( F=(f,G) \), be a
transformation whose base map \( f \) satisfies
Definition~\ref{def3} and whose fiber map \( G \)
satisfies the uniform fiberwise contraction condition
\textnormal{(H1)} stated in Section~2.2. For every $\mu \in \mathcal{AB}$ and a stable leaf $%
		\gamma \in \mathcal{F}^{s}$, it holds 
		\begin{equation}
			||\func{F}_{\gamma \ast }\mu |_{\gamma }||_{W}\leq ||\mu |_{\gamma }||_{W},
			\label{weak1}
		\end{equation}%
		where $F_{\gamma }:K\longrightarrow K$ is defined in Proposition \ref%
		{niceformulaab} and $\func{F}_{\gamma \ast }$ is the associated pushforward
		map. Moreover, if $\mu$ is a probability measure on $K$, it holds 
		\begin{equation}
			||\func{F}_{\gamma \ast}^n\mu ||_{W}=||\mu ||_{W}=1,\ \ \forall \ \
			n\geq 1.  \label{simples}
		\end{equation}
	\end{lemma}

	\begin{proposition}[Weak contraction on $\mathcal{L}^1$]Let \( F:\Sigma\to\Sigma \), \( F=(f,G) \), be a
transformation whose base map \( f \) satisfies
Definition~\ref{def3} and whose fiber map \( G \)
satisfies the uniform fiberwise contraction condition
\textnormal{(H1)} stated in Section~2.2. If $\mu \in \mathcal{L}^{1}$, then $||\func{F}_{\ast }\mu ||_{1}\leq ||\mu ||_{1}$.
		\label{weakcontral11234}
	\end{proposition}
	
	\begin{proposition}[Lasota-Yorke inequality]
		Let \( F:\Sigma\to\Sigma \), \( F=(f,G) \), be a
transformation whose base map \( f \) satisfies
Definition~\ref{def3} and whose fiber map \( G \)
satisfies the uniform fiberwise contraction condition
\textnormal{(H1)} stated in Section~2.2. Then, for all $\mu \in S^{1}$, it holds%
		\begin{equation}
			||\func{F}_{\ast }^{n}\mu ||_{S^{1}}\leq R_2r_2 ^{n}||\mu
			||_{S^{1}}+(C_2 +1)||\mu ||_{1},\ \ \forall n\geq 1,  \label{xx}
		\end{equation}where the constants $R_2, r_2$ and $C_2$ are from Corollary 2.16 of \cite{L}.
		\label{lasotaoscilation2}
	\end{proposition}

Let
\[
    \mathcal{V}_s
    :=\{\mu\in S^1:\mu(\Sigma)=0\}
\]
denote the subspace of signed measures in \( S^1 \)
with zero total mass.

\begin{proposition}[Exponential convergence to equilibrium]
	\label{5.8} Let \( F:\Sigma\to\Sigma \), \( F=(f,G) \), be a
transformation whose base map \( f \) satisfies
Definition~\ref{def3} and whose fiber map \( G \)
satisfies the uniform fiberwise contraction condition
\textnormal{(H1)} stated in Section~2.2. There exist $D_{2}\in \mathbb{R}$ and $0<\beta _{1}<1$ such that
	for every signed measure $\mu \in \mathcal{V}_{s}$, it holds 
	\begin{equation*}
		||\func{F}_{\ast }^{n}\mu ||_{1}\leq D_{2}\beta _{1}^{n}||\mu ||_{S^{1}},
	\end{equation*}%
	for all $n\geq 1$. \label{quasiquasiquasi}
\end{proposition}

\subsubsection{Key Properties of $\func{F}_*$}\label{bdhfdf}

\begin{theorem}\label{gfhduer}	Let \( F:\Sigma\to\Sigma \), \( F=(f,G) \), be a
transformation whose base map \( f \) satisfies
Definition~\ref{def3} and whose fiber map \( G \)
satisfies the uniform fiberwise contraction condition
\textnormal{(H1)} stated in Section~2.2. Then, it has a unique invariant probability in $S^{1}$.
\end{theorem}

\begin{theorem}[Spectral Gap for $\func{F}_*$ on $S^1$]
	\label{spgap}Let \( F:\Sigma\to\Sigma \), \( F=(f,G) \), be a
transformation whose base map \( f \) satisfies
Definition~\ref{def3} and whose fiber map \( G \)
satisfies the uniform fiberwise contraction condition
\textnormal{(H1)} stated in Section~2.2. Then, the operator $\func{F}_{\ast
	}:S^{1}\longrightarrow S^{1}$ can be written as 
	\begin{equation*}
		\func{F}_{\ast }=\func{P}+\func{N},
	\end{equation*}%
	where
	
	\begin{enumerate}
		\item[a)] $\func{P}$ is a projection, i.e., $\func{P}^2 = \func{P}$ and $\dim \operatorname{Im}(\func{P}) = 1$;

		\item[b)] there are $0\leq \lambda _0 <1$ and $U\geq0$ such that $\forall \mu \in S^1$ 
		\begin{equation*}
			||\func{N}^{n}(\mu )||_{S^{1}}\leq ||\mu||_{S^{1}} \lambda _0 ^{n}U;
		\end{equation*}
		
		\item[c)] $\func{P}\func{N}=\func{N}\func{P}=0$.
	\end{enumerate}
\end{theorem}

\section{Deterministic and Admissible Perturbations}\label{76235ghjdf}

In many dynamical systems of interest, the exact evolution rule is not fixed: small variations may occur due to changes in parameters, numerical approximations, or modeling choices. 
A central question is how such perturbations affect the statistical and ergodic properties of the system, in particular the existence, uniqueness, and stability of invariant measures. 
To address this, we introduce a general framework for describing and quantifying perturbations of the unperturbed map $F$.  
We focus on families of maps $\{F_\delta\}_{\delta \in [0,1)}$ depending on a small parameter $\delta$, where $\delta = 0$ corresponds to the original, unperturbed system. 
Our aim is to formalize the notion of \emph{$R(\delta)$-perturbation} and to distinguish those perturbations that preserve key functional-analytic properties needed for our statistical stability results.

\subsection{Deterministic Perturbations}

We first consider perturbations that are \emph{deterministic}, meaning that each value of the perturbation parameter $\delta$ determines a specific map $F_\delta$, without introducing any additional randomness.  
Such perturbations can arise, for example, from a smooth variation of parameters in a family of maps, or from a deterministic numerical scheme that approximates the original system.  

Our definition of deterministic $R(\delta)$-perturbation imposes precise quantitative conditions on how the perturbed system differs from the original one.  
These conditions involve:  
(i) preserving the underlying partition structure of the system,  
(ii) controlling the variation of inverse branches and derivatives in terms of a function $R(\delta)$, and  
(iii) ensuring that the perturbed system admits an invariant measure equivalent to the original one.  

In addition, we define the subclass of \emph{admissible} $R(\delta)$-perturbations, which are those satisfying further regularity and contraction properties.  
These properties guarantee the applicability of the functional-analytic techniques developed earlier in the paper, ensuring uniform control of the associated transfer operators and stable statistical behavior under perturbation.

\begin{definition}\label{rdelta}
	We say that a one-parameter family of measurable maps $\{F_\delta\}_{\delta \in [0,1)}$, with $F_0 = F$, is a \textbf{$R(\delta)$-perturbation of $F$} if there exists $\delta_0 \in (0,1)$ such that the following conditions hold:
	
	\begin{enumerate}
		\item [(U1)] For all $\delta \in [0,\delta_0)$, $f_\delta$ satisfies Definition \ref{def3} for the same partition $\mathcal{P}$. Moreover, we assume that $f_\delta(P_i) = f_0(P_i)$ for all $i$ and all $\delta \in [0,\delta_0)$.
		
		\item [(U2)] For every $\gamma \in I$ and for all $i \geq 1$, denote by $\gamma_{\delta,i}$ the $i$-th pre-image of $\gamma$ under $f_\delta$. Suppose there exists a function $\delta \mapsto R(\delta) \in \mathbb{R}^+$ such that
		\[
		\lim_{\delta \to 0^+} R(\delta) = 0,
		\] 
		and the following conditions hold:
		\begin{enumerate}
			\item [(U2.1)] 
			\[
			\sum_{i=1}^{\infty} \left| \frac{1}{f_\delta'(\gamma_{\delta,i})} - \frac{1}{f_0'(\gamma_{0,i})} \right| \leq R(\delta), \quad \forall \delta \in [0,\delta_0);
			\]
			\item [(U2.2)] 
			\[
			\esssup_\gamma \max_{i \geq 1} d_1(\gamma_{0,i}, \gamma_{\delta,i}) \leq R(\delta), \quad \forall \delta \in [0,\delta_0);
			\]
			\item [(U2.3)] For all $\delta \in [0,\delta_0)$, $G_0$ and $G_\delta$ are $R(\delta)$-close in the sup norm:
			\[
			d_2(G_0(x,y), G_\delta(x,y)) \leq R(\delta), \quad \forall (x,y) \in I \times K.
			\]
		\end{enumerate}

	\end{enumerate}
\end{definition}

\begin{definition}\label{ad}
	An \( R(\delta) \)-perturbation in the sense of
Definition~\ref{rdelta} is called \textbf{admissible}
if it satisfies the following additional conditions:
	\begin{enumerate}
		\item[(A1)] \textbf{Uniform Lasota--Yorke inequality.}
There exist constants \( D>0 \) and \( 0<\lambda<1 \),
independent of \( \delta \), such that, for all
\( h\in BV_m \), \( \delta\in[0,\delta_0) \), and
\( n\geq1 \),
		\[
		|\operatorname{P}_{f_\delta}^{n} h|_{v} \leq D \lambda^n |h|_{v} + D |h|_{1},
		\]
		where $\operatorname{P}_{f_\delta}$ denotes the Frobenius–Perron operator of $f_\delta$.

\end{enumerate}
Throughout the variation estimates below, the
reciprocal derivative of an iterate is understood
as the function
\[
    g_k(x):=
    \begin{cases}
        \dfrac{1}{|(f^k)'(x)|},
        & x\in\displaystyle\bigcup_i J_i,\\[4pt]
        0,
        & x\notin\displaystyle\bigcup_i J_i,
    \end{cases}
\]
where $\{J_i\}_i$ is the branch partition of $f^k$.
The notation $V_I(g_k)$ denotes the pointwise
variation of this specific representative,
including its boundary contributions.

  \begin{enumerate}
       
		\item[(A2)] \textbf{Uniform bounds for the variation estimates.}
There exists an integer $k\geq1$, independent of
$\delta \in [0, \delta_0)$, such that every $F_\delta$ satisfies
condition \textnormal{(H3)} with this same iterate. Let $\alpha_\delta$ denote the fiberwise contraction
rate of $F_\delta$ in condition \textnormal{(H1)}.
Assume that, on each interval of its branch
partition, $f_\delta^k$ is a $C^1$ diffeomorphism
onto its image.

Write
\[
    F_\delta^k=(f_\delta^k,G_{\delta,k}),
\]
and let $g_{\delta,k}$ denote the reciprocal
derivative of $f_\delta^k$, extended by zero
outside the interiors of its branches, according
to the convention above. Set
\[
    q_{k,\delta}:=\|g_{\delta,k}\|_{L^\infty(m)},
    \qquad
    \alpha_{4,\delta}:=\alpha_\delta^k q_{k,\delta},
\]
and
\[
    U_{4,\delta}
    :=|G_{\delta,k}|_{\mathrm{Lip}}q_{k,\delta}
      +V_I(g_{\delta,k}),
\]
where the horizontal Lipschitz constant is taken
with respect to the branch partition of
$f_\delta^k$. Assume that
\[
    \sup_{\delta\in[0,\delta_0)}
    \alpha_{4,\delta}<1,
    \qquad
    \sup_{\delta\in[0,\delta_0)}
    U_{4,\delta}<\infty.
\]
	\end{enumerate}
\end{definition}

\begin{lemma}\label{nslfdflsdjlf}
Let $\{F_\delta\}_{\delta\in[0,1)}$ be an admissible
$R(\delta)$-perturbation in the sense of
Definition~\ref{ad}. For each $i\geq1$, write
\[
    f_{\delta,i}:=f_\delta|_{I_i},
    \qquad
    J_{\delta,i}:=f_{0,i}(I_i)\cap f_{\delta,i}(I_i).
\]
For $\gamma\in J_{\delta,i}$, define
\[
    \gamma_{0,i}:=f_{0,i}^{-1}(\gamma),
    \qquad
    \gamma_{\delta,i}:=f_{\delta,i}^{-1}(\gamma).
\]
Then, for every probability measure
$\mu\in\mathcal{BV}_m$, every
$\delta\in[0,\delta_0)$, and every $i\geq1$,
\[
\begin{aligned}
    &\left\|
        (F_{0,\gamma_{0,i}})_*
        \bigl(\mu|_{\gamma_{0,i}}\bigr)
        -
        (F_{\delta,\gamma_{\delta,i}})_*
        \bigl(\mu|_{\gamma_{0,i}}\bigr)
    \right\|_W
    \leq
    R(\delta)\bigl(1+|G_0|_{\mathrm{Lip}}\bigr)
    \|\mu|_{\gamma_{0,i}}\|_W
\end{aligned}
\]
for $m$-almost every $\gamma\in J_{\delta,i}$,
where
\[
    F_{\delta,x}:K\to K,
    \qquad F_{\delta,x}(y):=G_\delta(x,y),
\]
is the fiber map defined in
\eqref{ritiruwt}.
\end{lemma}\begin{proof}
Fix $\delta\in[0,\delta_0)$ and a branch index $i$.
For almost every $\gamma$ for which both inverse
branches are defined, write
\[
    x_0:=\gamma_{0,i},
    \qquad
    x_\delta:=\gamma_{\delta,i},
    \qquad
    \eta:=\mu|_{x_0}.
\]
The points $x_0$ and $x_\delta$ belong to the same
partition interval. Hence, by \textnormal{(H2)},
\textnormal{(U2.2)}, and \textnormal{(U2.3)},
\[
\begin{aligned}
    d_2(G_0(x_0,y),G_\delta(x_\delta,y))
    &\leq
    d_2(G_0(x_0,y),G_0(x_\delta,y))\\
    &\quad+
    d_2(G_0(x_\delta,y),G_\delta(x_\delta,y))\\
    &\leq
    |G_0|_{\mathrm{Lip}}|x_0-x_\delta|
    +R(\delta)\\
    &\leq
    (1+|G_0|_{\mathrm{Lip}})R(\delta).
\end{aligned}
\]
For every test function $h$ with
$\operatorname{Lip}(h)\leq1$ and
$\|h\|_\infty\leq1$, positivity of $\eta$ gives
\[
\begin{aligned}
    &\left|
    \int_K h(G_0(x_0,y))\,d\eta(y)
    -
    \int_K h(G_\delta(x_\delta,y))\,d\eta(y)
    \right|\\
    &\qquad\leq
    \int_K
    d_2(G_0(x_0,y),G_\delta(x_\delta,y))
    \,d\eta(y)\\
    &\qquad\leq
    (1+|G_0|_{\mathrm{Lip}})R(\delta)\eta(K).
\end{aligned}
\]
Since $\eta(K)=\|\eta\|_W$, taking the supremum
over $h$ proves the assertion.
\end{proof}For the next, proposition and henceforth, for a given path $\Gamma _\mu ^\omega \in \Gamma_{ \mu }$ (associated with the disintegration $\omega = (\{\mu _\gamma\}_\gamma, \phi _1)$, of $\mu$), unless written otherwise, we consider the particular path $\Gamma_{\func{F_*}\mu} ^\omega \in \Gamma_{\func{F_*}\mu}$ defined in Remark \ref{çlkhgh} by the expression

\begin{equation}
	\Gamma_{\func{F_*}\mu} ^\omega (\gamma)=\sum_{i=1}^{\infty}{\func{F}%
		_{\gamma _i \ast }\Gamma _\mu ^\omega (\gamma_i)g_i(\gamma _i)}\ \ m%
	\mathnormal{-a.e.}\ \ \gamma \in I.  \label{niceformulaaareer}
\end{equation}Recall that
\[
    \Gamma_\mu^\omega(\gamma)
    =\mu|_\gamma
    =\pi_{2*}\bigl(\phi_1(\gamma)\mu_\gamma\bigr),
    \qquad
    \phi_1=\frac{d\pi_{1*}\mu}{dm}.
\]
The marginal of $F_*\mu$ has density $P_f\phi_1$,
whereas its fiber restrictions are given by
\eqref{niceformulaaareer}.
\begin{lemma}[One-step variation estimate]
\label{kjdfhkkhfdjfht}
Let $F=(f,G)$ satisfy condition \textnormal{(H3)}, and
write
\[
    \overline F:=F^k=(\overline f,G_k),
    \qquad \overline f:=f^k,
\]
where $k$ is the iterate appearing in that condition.

Let $\{J_i=(a_i,b_i)\}_{i\geq1}$ be a countable family
of pairwise disjoint open intervals whose union has
full Lebesgue measure in $I$. Assume that each branch
\[
    \overline f_i:=\overline f|_{J_i}
\]
is a $C^1$ diffeomorphism onto its image. Set
\[
    g_i(x):=\frac{1}{|\overline f_i'(x)|},
    \qquad x\in J_i,
\]
and assume that
\[
    q_k:=\sup_{i\geq1}\|g_i\|_{L^\infty}<\infty
\]
and
\[
    \mathcal B_k
    :=
    \sum_{i\geq1}
    \left(
        V_{J_i}(g_i)
        +g_i(a_i^+)+g_i(b_i^-)
    \right)
    <\infty.
\]
Here $V_{J_i}(g_i)$ denotes (the pointwise) variation
on $J_i$, and the endpoint values are the
corresponding one-sided limits.

Let
\[
    L_k:=|G_k|_{\mathrm{Lip}},
    \qquad
    \alpha_4:=\alpha^k q_k<1,
    \qquad
    U_4:=L_kq_k+\mathcal B_k.
\]
Then, for every positive measure
$\mu\in\mathcal{BV}_m$, its marginal density
\[
    \phi_1:=\frac{d\pi_{1*}\mu}{dm}
\]
belongs to $BV_m$, with $V(\phi_1)\leq V(\mu)$,
and $\overline F_*\mu\in\mathcal{BV}_m$. Moreover,
\begin{equation}\label{erkjwr166}
    V(\overline F_*\mu)
    \leq
    \alpha_4 V(\mu)
    +(q_k+U_4)V(\phi_1)
    +U_4\|\mu\|_1.
\end{equation}
\end{lemma}

\begin{proof}
We first record an estimate for the action of a
contracting map on measures of possibly different
masses. If $T:K\to K$ is $\alpha^k$-Lipschitz and
$\eta$ is a finite signed measure on $K$, then
\begin{equation}\label{eq:fiber-mass-correction}
    \|T_*\eta\|_W
    \leq
    \alpha^k\|\eta\|_W+|\eta(K)|.
\end{equation}
Indeed, fix $z_0\in K$ and a test function $h$ with
$\operatorname{Lip}(h)\leq1$ and $\|h\|_\infty\leq1$.
The function
\[
    \psi(z):=h(Tz)-h(Tz_0)
\]
satisfies
\[
    \operatorname{Lip}(\psi)\leq\alpha^k,
    \qquad
    \|\psi\|_\infty\leq\alpha^k,
\]
because $\operatorname{diam}(K)=1$. Thus
\[
    \left|\int_K h\,d(T_*\eta)\right|
    \leq
    \alpha^k\|\eta\|_W+|\eta(K)|.
\]
When $\alpha=0$, the same conclusion follows from
$\psi=0$. Taking the supremum over $h$ proves
\eqref{eq:fiber-mass-correction}.

We next examine the marginal. For any path
$\Gamma$ representing $\mu$, on its full measure
domain of definition we have
\[
    \Gamma(x)(K)=\phi_1(x),
    \qquad
    \|\Gamma(x)\|_W=\phi_1(x),
\]
since $\mu$ is positive. Moreover,
\[
    |\phi_1(x)-\phi_1(y)|
    \leq \|\Gamma(x)-\Gamma(y)\|_W,
\]
because the constant function $1$ is admissible
in the definition of $\|\cdot\|_W$.

A scalar function of bounded variation on a full
measure subset of $I$ extends to $I$, using
one-sided limits at missing points, without
increasing its variation. Consequently, taking
variations and then the infimum over paths gives
\begin{equation}\label{eq:marginal-var-bound}
    \phi_1\in BV_m,
    \qquad
    V(\phi_1)\leq V(\mu).
\end{equation}
Set
\[
    s:=\|\phi_1\|_{L^\infty(m)}.
\]
The elementary embedding of $BV_m$ into
$L^\infty(m)$ on $I=[0,1]$ gives
\begin{equation}\label{eq:marginal-sup-bound}
    s\leq\|\phi_1\|_{L^1(m)}+V(\phi_1).
\end{equation}
For completeness, for any representative $u$ of
$\phi_1$ with finite pointwise variation,
\[
    |u(x)|\leq |u(t)|+V_I(u).
\]
Integrating with respect to $t$, taking the
essential supremum in $x$, and then taking the
infimum over such representatives yields
\eqref{eq:marginal-sup-bound}.

Fix $\varepsilon>0$. Choose a path $\Gamma$
representing $\mu$ and a scalar representative
of $\phi_1$ such that
\[
    V_E(\Gamma)\leq V(\mu)+\varepsilon,
    \qquad
    V_E(\phi_1)\leq V(\phi_1)+\varepsilon
\]
on a common full measure set $E\subset I$.
By removing a null set, we may also assume that
$\Gamma(x)$ is positive,
\[
    \Gamma(x)(K)=\phi_1(x)\leq s,
\]
and that the fiberwise contraction condition for
$G_k$ holds at every $x\in E$.

For $x\in E$, define
\[
    T_x:=G_k(x,\cdot),
    \qquad
    A(x):=(T_x)_*\Gamma(x).
\]
If $u,v\in E\cap J_i$, then
\[
\begin{aligned}
    A(u)-A(v)
    &=(T_u)_*(\Gamma(u)-\Gamma(v))\\
    &\quad +(T_u)_*\Gamma(v)-(T_v)_*\Gamma(v).
\end{aligned}
\]
By \eqref{eq:fiber-mass-correction}, the first
term satisfies
\[
    \|(T_u)_*(\Gamma(u)-\Gamma(v))\|_W
    \leq
    \alpha^k\|\Gamma(u)-\Gamma(v)\|_W
    +|\phi_1(u)-\phi_1(v)|.
\]
For the second term, positivity and the horizontal
Lipschitz condition give
\[
    \|(T_u)_*\Gamma(v)-(T_v)_*\Gamma(v)\|_W
    \leq L_k|u-v|\phi_1(v)
    \leq L_k s|u-v|.
\]
Therefore,
\[
\begin{aligned}
    \|A(u)-A(v)\|_W
    &\leq
    \alpha^k\|\Gamma(u)-\Gamma(v)\|_W\\
    &\quad +|\phi_1(u)-\phi_1(v)|
    +L_ks|u-v|.
\end{aligned}
\]
Summing over finite ordered sets in $E\cap J_i$
yields
\begin{equation}\label{eq:branch-fiber-variation}
\begin{aligned}
    V_{E\cap J_i}(A)
    &\leq
    \alpha^k V_{E\cap J_i}(\Gamma)
    +V_{E\cap J_i}(\phi_1)
    +L_ks\,m(J_i).
\end{aligned}
\end{equation}

Now define
\[
    Q_i(x):=g_i(x)A(x),
    \qquad x\in E\cap J_i.
\]
Since $\|A(x)\|_W=\phi_1(x)\leq s$, the product
estimate for variation gives
\[
    V_{E\cap J_i}(Q_i)
    \leq
    q_k V_{E\cap J_i}(A)+sV_{J_i}(g_i).
\]
Combining this with
\eqref{eq:branch-fiber-variation}, we obtain
\begin{equation}\label{eq:weighted-branch-variation}
\begin{aligned}
    V_{E\cap J_i}(Q_i)
    &\leq
    \alpha^kq_k V_{E\cap J_i}(\Gamma)
    +q_kV_{E\cap J_i}(\phi_1)\\
    &\quad
    +L_kq_ks\,m(J_i)+sV_{J_i}(g_i).
\end{aligned}
\end{equation}

Write $v_i:=\overline f_i^{-1}$. By the
disintegration formula in
Proposition~\ref{niceformulaab}, applied to
$\overline F$, a path representing
$\overline F_*\mu$ is
\[
    \Gamma'(x)=\sum_{i\geq1}H_i(x),
\]
where
\[
    H_i(x):=
    \begin{cases}
        Q_i(v_i(x)),
        &x\in\overline f_i(J_i),\\
        0,
        &x\notin\overline f_i(J_i).
    \end{cases}
\]
These expressions are considered on a common
full measure set $E'\subset I$ on which all
relevant inverse images belong to $E$.
Such a set exists because each $C^1$ branch
$\overline f_i$ maps null sets to null sets.
We also remove the endpoints of all branch images
from $E'$.

Composition with the monotone map $v_i$ does
not increase variation. Extending the resulting
path by zero outside $\overline f_i(J_i)$ adds
at most the two boundary contributions
\[
    s\,g_i(a_i^+)
    \quad\text{and}\quad
    s\,g_i(b_i^-).
\]
To see this, insert points approaching the
endpoints from inside the branch in any finite
variation sum and use
\[
    \|Q_i(x)\|_W=g_i(x)\phi_1(x)\leq s\,g_i(x).
\]
Hence
\begin{equation}\label{eq:zero-extension-variation}
\begin{aligned}
    V_{E'}(H_i)
    &\leq
    \alpha^kq_kV_{E\cap J_i}(\Gamma)
    +q_kV_{E\cap J_i}(\phi_1)\\
    &\quad +L_kq_ks\,m(J_i)\\
    &\quad
    +s\left(
        V_{J_i}(g_i)
        +g_i(a_i^+)+g_i(b_i^-)
    \right).
\end{aligned}
\end{equation}

We justify summation over the countably many
branches. Since $g_i\geq0$,
\[
    2\sup_{J_i}g_i
    \leq
    V_{J_i}(g_i)+g_i(a_i^+)+g_i(b_i^-).
\]
Indeed, this follows by estimating the variation
along points approaching $a_i$, an arbitrary
point of $J_i$, and points approaching $b_i$.
Consequently,
\[
    \sum_{i\geq1}
    \sup_{x\in E'}\|H_i(x)\|_W
    \leq \frac{s\mathcal B_k}{2}<\infty.
\]
Each $H_i(x)$ is positive, so the series defines
a finite positive measure and converges uniformly
on $E'$ in $\|\cdot\|_W$. The triangle inequality
on finite ordered sets therefore implies
\[
    V_{E'}(\Gamma')
    \leq\sum_{i\geq1}V_{E'}(H_i).
\]

The intervals $J_i$ are pairwise disjoint, so
\[
    \sum_{i\geq1}V_{E\cap J_i}(\Gamma)
    \leq V_E(\Gamma),
    \qquad
    \sum_{i\geq1}V_{E\cap J_i}(\phi_1)
    \leq V_E(\phi_1).
\]
Moreover, $\sum_{i\geq1}m(J_i)=1$.
Summing \eqref{eq:zero-extension-variation}
thus gives
\[
    V_{E'}(\Gamma')
    \leq
    \alpha_4V_E(\Gamma)
    +q_kV_E(\phi_1)+U_4s.
\]
Since $\Gamma'$ represents $\overline F_*\mu$,
\[
    V(\overline F_*\mu)
    \leq
    \alpha_4\bigl(V(\mu)+\varepsilon\bigr)
    +q_k\bigl(V(\phi_1)+\varepsilon\bigr)
    +U_4s.
\]
Letting $\varepsilon\to0$, we obtain
\[
    V(\overline F_*\mu)
    \leq
    \alpha_4V(\mu)+q_kV(\phi_1)
    +U_4\|\phi_1\|_{L^\infty(m)}.
\]
Finally, apply \eqref{eq:marginal-sup-bound} and
use positivity to write
\[
    \|\phi_1\|_{L^1(m)}=\mu(\Sigma)=\|\mu\|_1.
\]
This proves \eqref{erkjwr166}, whose finite
right-hand side also shows that
$\overline F_*\mu\in\mathcal{BV}_m$.
\end{proof}

\begin{remark}
Let
$\{J_{\delta,i}=(a_{\delta,i},b_{\delta,i})\}_{i\geq1}$
be the branch partition of $f_\delta^k$.
Since $g_{\delta,k}$ is nonnegative and is defined
to be zero outside the interiors of these branches,
its pointwise variation satisfies
\[
\begin{aligned}
    \mathcal B_{k,\delta}
    &:=
    \sum_{i\geq1}
    \Big(
        V_{J_{\delta,i}}(g_{\delta,k})
        +g_{\delta,k}(a_{\delta,i}^+)
        +g_{\delta,k}(b_{\delta,i}^-)
    \Big)\\
    &=V_I(g_{\delta,k}).
\end{aligned}
\]
Consequently, condition \textnormal{(A2)} of
Definition~\ref{ad} implies
\[
    \sup_{\delta\in[0,\delta_0)}
    \mathcal B_{k,\delta}
    \leq
    \sup_{\delta\in[0,\delta_0)}
    U_{4,\delta}
    <\infty.
\]
Thus, the boundary contributions in the one-step
variation estimate are already controlled by
the admissibility assumption.

For a single map $f$ satisfying
Definition~\ref{def1}, finiteness of $\mathcal B_k$
also follows directly from the assumptions on
the base map. Indeed, set
\[
    S:=\sum_{i\geq1}\frac{1}{f_i'(a_i)}<\infty.
\]
Each branch is increasing and convex, so its
reciprocal derivative $g_i=1/f_i'$ is
non-increasing. Hence
\[
    V_{I_i}(g_i)+g_i(a_i^+)+g_i(b_i^-)
    =2g_i(a_i^+),
\]
and therefore $\mathcal B_1=2S$.

More generally, each nonempty cylinder of $f^k$
with itinerary $(i_0,\ldots,i_{k-1})$ is an
interval on which $f^k$ is increasing and convex.
The chain rule gives
\[
    \sup_{x\text{ in the cylinder}}
    \frac{1}{|(f^k)'(x)|}
    \leq
    \prod_{j=0}^{k-1}\frac{1}{f_{i_j}'(a_{i_j})}.
\]
Applying the preceding variation identity on
each cylinder and summing yields
\[
    \mathcal B_k
    \leq
    2\sum_{i_0,\ldots,i_{k-1}}
    \prod_{j=0}^{k-1}\frac{1}{f_{i_j}'(a_{i_j})}
    =2S^k<\infty.
\]
For a family of such maps, this argument gives
a uniform bound whenever the corresponding
sums $S_\delta$ are uniformly bounded.
\end{remark}
\begin{remark}\label{riirorpdf}
	Let $\nu$ be the probability measure defined as the product $\nu = m \times \nu_2$, where $m$ denotes the Lebesgue measure on $I$ and $\nu_2$ is a fixed probability measure on $K$. Consider its trivial disintegration $\omega_0 = (\{\nu_{\gamma}\}_{\gamma}, \phi_1)$, where $\nu_{\gamma} = \func{\pi_{2,\gamma}^{-1}{_*}}\nu_2$ for all $\gamma$ and $\phi_1 \equiv 1$. By this definition, we have
	\begin{equation*}
		\nu|_{\gamma} = \nu_2, \quad \forall\, \gamma.
	\end{equation*}
	In other words, the path $\Gamma^{\omega_0}_{\nu}$ is constant: $\Gamma^{\omega_0}_{\nu}(\gamma) = \nu_2$ for all $\gamma$. Hence, $\nu \in \mathcal{BV}_m$. Moreover, for each $n \in \mathbb{N}$, let $\omega_n$ be the disintegration of the measure $\func{\overline{F}}_*^n \nu$ obtained from $\omega_0$ by applying Proposition~\ref{niceformulaab}. Then, by a simple induction, we can consider the path $\Gamma^{\omega_n}_{\func{\overline{F}}_*^n \nu}$ associated with $\omega_n$. This path will be used in the proof of Theorem \ref{thshgf}.
\end{remark}

\begin{lemma}\label{UF2ass} 
	Let \( \{F_\delta\}_{\delta\in[0,1)} \) be an admissible
\( R(\delta) \)-perturbation in the sense of
Definition~\ref{ad}. Denote by $\func{F_\delta{_\ast}}$ their transfer operators, and by $\mu_{\delta}$ their fixed points (probability measures) on $S^1$. Suppose that the family $\{\mu_{\delta}\}_{\delta \in [0,1)}$ satisfies 
	\[
	V(\mu_{\delta}) \leq B_u, \quad \forall\, \delta \in [0,\delta_0).
	\]
	Then there exists a constant $C_{1}$ such that
	\[
	\|(\func{F_0{_\ast}}-\func{F_\delta{_\ast}})\mu_{\delta}\|_{1} \leq C_{1} R(\delta),
	\]
	for all $\delta \in [0,\delta_0)$, where $C_1 := |G_0|_{\lip} + 3B_u + 2$.
\end{lemma}

\begin{proof}
	We start by writing
\begin{equation}\label{12112}
	\|(\func{F_0{_\ast}}-\func{F_\delta{_\ast}})\mu_{\delta}\|_{1}
	= \int_{I} \|(\func{F_0{_\ast}}\mu_{\delta})|_{\gamma} - (\func{F_\delta{_\ast}}\mu_{\delta})|_{\gamma}\|_{W} \, dm.
\end{equation} 

Let $f_{\delta,i} := f_{\delta}|_{P_{i}}$ denote the branches of $f_{\delta}$ defined on $P_{i} \in \mathcal{P}$, for $1 \leq i < \infty$. Recall that we write $\gamma_{\delta, i} := f^{-1}_{\delta, i}(\gamma)$ whenever $\gamma \in f_\delta(P_i)=f_0(P_i)$, and that the inverse-branch proximity condition
\textnormal{(U2.2)} from Definition~\ref{rdelta}
gives the following bound:  
\begin{equation}\label{tyuryt}
d_1(\gamma_{0,i},\gamma_{\delta,i}) \leq R(\delta),
\quad \text{for }m\text{-a.e. }\gamma\in f_\delta(P_i),
\quad \forall\, i \geq 1.
\end{equation}
For the remainder of this proof, fix \( \delta\in[0,\delta_0) \)
and write \( \mu:=\mu_\delta \).
We use the notation
\( F_{\delta,\gamma_{\delta,i}}
:=F_{\delta,f^{-1}_{\delta,i}(\gamma)} \)
for the fiber map over the \( i \)-th inverse image.
The disintegration formula from
Proposition~\ref{niceformulaab} then gives
\begin{equation}\label{asdf}
	(\func{F_0{_\ast}}\mu - \func{F_\delta{_\ast}}\mu)|_{\gamma}
	= \sum_{i=1}^{\infty} \frac{(\func{F}_{0,\gamma_{0,i}}{_\ast})\,\mu|_{\gamma_{0,i}}}{|f'_{0}(\gamma_{0,i})|}
	- \sum_{i=1}^{\infty} \frac{(\func{F}_{\delta,\gamma_{\delta,i}}{_\ast})\,\mu|_{\gamma_{\delta,i}}}{|f'_{\delta}(\gamma_{\delta,i})|},
	\quad m\text{-a.e. } \gamma \in I,
\end{equation}
where each summand is understood to vanish outside
the corresponding branch image. By definition, it also holds that
\begin{align*}
	(\func{F_0{_\ast}}\mu - \func{F_\delta{_\ast}}\mu)|_{\gamma}
	&= \sum_{i=1}^{\infty} \frac{(\func{F}_{0,\gamma_{0,i}}{_\ast})\,\mu|_{\gamma_{0,i}}\,\chi_{f_0(P_i)}(\gamma)}{|f'_{0}(\gamma_{0,i})|} \\
	&\quad - \sum_{i=1}^{\infty} \frac{(\func{F}_{\delta,\gamma_{\delta,i}}{_\ast})\,\mu|_{\gamma_{\delta,i}}\,\chi_{f_\delta(P_i)}(\gamma)}{|f'_{\delta}(\gamma_{\delta,i})|}.
\end{align*}
By condition \textnormal{(U1)} of Definition~\ref{rdelta},
$f_\delta(P_i)=f_0(P_i)$ for all $i$ and all
$\delta\in[0,\delta_0)$. Thus, the characteristic functions
agree and the terms outside $f_\delta(P_i)$ vanish.
This observation will be used when splitting the sums below.
Thus, although we work with \eqref{asdf}, we implicitly use
the equivalent form containing $\chi_{f_\delta(P_i)}$.

We now split:
\begin{equation*}
	\|(\func{F_0{_\ast}}\mu-\func{F_\delta{_\ast}}\mu)|_\gamma\|_W
	\leq \func{I}(\gamma)+\func{II}(\gamma),
\end{equation*}
where
\begin{equation}\label{I}
	\func{I}(\gamma)
	:= \left\|
		\sum_{i=1}^{+\infty}
		\frac{{\func{F}_{0,\gamma_{0,i}}{_\ast}}\mu|_{\gamma_{0,i}}}
		{|f'_0(\gamma_{0,i})|}
		-
		\sum_{i=1}^{+\infty}
		\frac{{\func{F}_{\delta,\gamma_{\delta,i}}{_\ast}}\mu|_{\gamma_{0,i}}}
		{|f'_\delta(\gamma_{\delta,i})|}
	\right\|_W
\end{equation}
and
\begin{equation}\label{II}
	\func{II}(\gamma)
	:= \left\|
		\sum_{i=1}^{+\infty}
		\frac{{\func{F}_{\delta,\gamma_{\delta,i}}{_\ast}}\mu|_{\gamma_{0,i}}}
		{|f'_\delta(\gamma_{\delta,i})|}
		-
		\sum_{i=1}^{+\infty}
		\frac{{\func{F}_{\delta,\gamma_{\delta,i}}{_\ast}}\mu|_{\gamma_{\delta,i}}}
		{|f'_\delta(\gamma_{\delta,i})|}
	\right\|_W.
\end{equation}
	
Let us estimate $\func{I}$ of equation (\ref{I}).
By applying the triangle inequality, we have
\[
    \func{I}(\gamma)
    \leq \func{I}_a(\gamma)+\func{I}_b(\gamma),
\]
where
\begin{equation}
    \func{I}_a(\gamma) :=
    \left\|
        \sum_{i=1}^{+\infty}
        \frac{{\func{F}_{0,\gamma_{0,i}}{_\ast}}\mu|_{\gamma_{0,i}}}
        {|f'_0(\gamma_{0,i})|}
        -
        \sum_{i=1}^{+\infty}
        \frac{{\func{F}_{\delta,\gamma_{\delta,i}}{_\ast}}\mu|_{\gamma_{0,i}}}
        {|f'_0(\gamma_{0,i})|}
    \right\|_W
\end{equation}
and
\begin{equation}
    \func{I}_b(\gamma) :=
    \left\|
        \sum_{i=1}^{+\infty}
        \frac{{\func{F}_{\delta,\gamma_{\delta,i}}{_\ast}}\mu|_{\gamma_{0,i}}}
        {|f'_0(\gamma_{0,i})|}
        -
        \sum_{i=1}^{+\infty}
        \frac{{\func{F}_{\delta,\gamma_{\delta,i}}{_\ast}}\mu|_{\gamma_{0,i}}}
        {|f'_\delta(\gamma_{\delta,i})|}
    \right\|_W.
\end{equation}
To estimate \(\func{I}_a(\gamma)\), we apply the
fiber-map perturbation bound from
Lemma~\ref{nslfdflsdjlf}, obtaining
\begin{eqnarray*}
    \func{I}_a(\gamma)
    &\leq&
    \sum_{i=1}^{+\infty}
    \left\|
        \frac{{\func{F}_{0,\gamma_{0,i}}{_\ast}}\mu|_{\gamma_{0,i}}}
        {|f'_0(\gamma_{0,i})|}
        -
        \frac{{\func{F}_{\delta,\gamma_{\delta,i}}{_\ast}}\mu|_{\gamma_{0,i}}}
        {|f'_0(\gamma_{0,i})|}
    \right\|_W
    \\
    &\leq&
    \sum_{i=1}^{+\infty}
    \frac{
        \left\|
            ({\func{F}_{0,\gamma_{0,i}}{_\ast}}
            -{\func{F}_{\delta,\gamma_{\delta,i}}{_\ast}})
            \mu|_{\gamma_{0,i}}
        \right\|_W
    }{|f'_0(\gamma_{0,i})|}
    \\
    &\leq&
    \sum_{i=1}^{+\infty}
    \frac{
        R(\delta)(1+|G_0|_{\lip})
        \|\mu|_{\gamma_{0,i}}\|_W
    }{|f'_0(\gamma_{0,i})|}
    \\
    &=&
    R(\delta)(1+|G_0|_{\lip})
    \sum_{i=1}^{+\infty}
    \frac{|\phi_{1,\delta}(\gamma_{0,i})|}
    {|f'_0(\gamma_{0,i})|}
    \\
    &=&
    R(\delta)(1+|G_0|_{\lip})
    \func{P}_{f_0}(|\phi_{1,\delta}|)(\gamma),
\end{eqnarray*}
and integrating yields
\begin{equation}\label{ia}
\begin{aligned}
    \int \func{I}_a(\gamma)\,dm(\gamma)
    &\leq
    R(\delta)(1+|G_0|_{\lip})
    \int \func{P}_{f_0}(|\phi_{1,\delta}|)(\gamma)\,dm(\gamma)\\
    &=R(\delta)(1+|G_0|_{\lip}).
\end{aligned}
\end{equation}

	To estimate \( \func{I}_b(\gamma) \), we use the
reciprocal-derivative bound \textnormal{(U2.1)}
from Definition~\ref{rdelta}, the non-expansion
of the fiber pushforward operators from
Lemma~\ref{niceformulaac}, and the identity
\( \|\mu|_\gamma\|_W=|\phi_{1,\delta}(\gamma)| \).
We obtain
\begin{eqnarray*}
    \func{I}_b(\gamma)
    &\leq&
    \sum_{i=1}^{+\infty}
    \left\|
        \frac{{\func{F}_{\delta,\gamma_{\delta,i}}{_\ast}}\mu|_{\gamma_{0,i}}}
        {|f'_0(\gamma_{0,i})|}
        -
        \frac{{\func{F}_{\delta,\gamma_{\delta,i}}{_\ast}}\mu|_{\gamma_{0,i}}}
        {|f'_\delta(\gamma_{\delta,i})|}
    \right\|_W
    \\
    &\leq&
    \sum_{i=1}^{+\infty}
    \left|
        \frac{1}{|f'_0(\gamma_{0,i})|}
        -
        \frac{1}{|f'_\delta(\gamma_{\delta,i})|}
    \right|
    \left\|
        {\func{F}_{\delta,\gamma_{\delta,i}}{_\ast}}
        \mu|_{\gamma_{0,i}}
    \right\|_W
    \\
    &\leq&
    \sum_{i=1}^{+\infty}
    \left|
        \frac{1}{|f'_0(\gamma_{0,i})|}
        -
        \frac{1}{|f'_\delta(\gamma_{\delta,i})|}
    \right|
    \|\mu|_{\gamma_{0,i}}\|_W
    \\
    &=&
    \sum_{i=1}^{+\infty}
    \left|
        \frac{1}{|f'_0(\gamma_{0,i})|}
        -
        \frac{1}{|f'_\delta(\gamma_{\delta,i})|}
    \right|
    |\phi_{1,\delta}(\gamma_{0,i})|
    \\
    &\leq&
    |\phi_{1,\delta}|_\infty
    \sum_{i=1}^{+\infty}
    \left|
        \frac{1}{|f'_0(\gamma_{0,i})|}
        -
        \frac{1}{|f'_\delta(\gamma_{\delta,i})|}
    \right|
    \\
    &\leq&
    (V(\mu)+1)
    \sum_{i=1}^{+\infty}
    \left|
        \frac{1}{|f'_0(\gamma_{0,i})|}
        -
        \frac{1}{|f'_\delta(\gamma_{\delta,i})|}
    \right|
    \\
    &\leq&
    (B_u+1)
    \sum_{i=1}^{+\infty}
    \left|
        \frac{1}{|f'_0(\gamma_{0,i})|}
        -
        \frac{1}{|f'_\delta(\gamma_{\delta,i})|}
    \right|.
\end{eqnarray*}
By the reverse triangle inequality,
\[
    \left|
        \frac{1}{|f'_0(\gamma_{0,i})|}
        -
        \frac{1}{|f'_\delta(\gamma_{\delta,i})|}
    \right|
    \leq
    \left|
        \frac{1}{f'_0(\gamma_{0,i})}
        -
        \frac{1}{f'_\delta(\gamma_{\delta,i})}
    \right|.
\]
Integrating and applying the reciprocal-derivative
bound \textnormal{(U2.1)} from
Definition~\ref{rdelta}, we obtain
\begin{equation}\label{ib}
    \int \func{I}_b(\gamma)\,dm(\gamma)
    \leq (B_u+1)R(\delta).
\end{equation}To estimate \( \func{II}(\gamma) \), we use the
non-expansion of the fiber pushforward operators
in the Wasserstein-type norm, established in
Lemma~\ref{niceformulaac}. Thus,
\begin{eqnarray*}
    \func{II}(\gamma)
    &\leq&
    \sum_{i=1}^{+\infty}
    \left\|
        \frac{{\func{F}_{\delta,\gamma_{\delta,i}}{_\ast}}\mu|_{\gamma_{0,i}}}
        {|f'_\delta(\gamma_{\delta,i})|}
        -
        \frac{{\func{F}_{\delta,\gamma_{\delta,i}}{_\ast}}\mu|_{\gamma_{\delta,i}}}
        {|f'_\delta(\gamma_{\delta,i})|}
    \right\|_W
    \\
    &\leq&
    \sum_{i=1}^{+\infty}
    \frac{1}{|f'_\delta(\gamma_{\delta,i})|}
    \left\|
        {\func{F}_{\delta,\gamma_{\delta,i}}{_\ast}}
        \bigl(\mu|_{\gamma_{0,i}}-\mu|_{\gamma_{\delta,i}}\bigr)
    \right\|_W
    \\
    &\leq&
    \sum_{i=1}^{+\infty}
    \frac{1}{|f'_\delta(\gamma_{\delta,i})|}
    \left\|
        \mu|_{\gamma_{0,i}}-\mu|_{\gamma_{\delta,i}}
    \right\|_W.
\end{eqnarray*}Using the inverse-branch proximity bound
\eqref{tyuryt}, which follows from
\textnormal{(U2.2)} in Definition~\ref{rdelta},
we integrate and apply the change of variables
formula on each branch to obtain
\begin{eqnarray*}
    \int \func{II}(\gamma)\,dm(\gamma)
    &\leq&
    \int \sum_{i=1}^{+\infty}
    \frac{
        \|\mu|_{\gamma_{0,i}}-\mu|_{\gamma_{\delta,i}}\|_W
    }{|f'_\delta(\gamma_{\delta,i})|}
    \,dm(\gamma)
    \\
    &\leq&
    \sum_{i=1}^{+\infty}
    \int_{f_\delta(P_i)}
    \frac{1}{|f'_\delta(\gamma_{\delta,i})|}
    \esssup_{\gamma_1,\gamma_2\in B(\gamma_{\delta,i};R(\delta))\cap I}
    \|\mu|_{\gamma_1}-\mu|_{\gamma_2}\|_W
    \,dm(\gamma)
    \\
    &=&
    \sum_{i=1}^{+\infty}
    \int_{P_i}
    \esssup_{\gamma_1,\gamma_2\in B(\gamma;R(\delta))\cap I}
    \|\mu|_{\gamma_1}-\mu|_{\gamma_2}\|_W
    \,dm(\gamma)
    \\
    &=&
    \int_I
    \esssup_{\gamma_1,\gamma_2\in B(\gamma;R(\delta))\cap I}
    \|\mu|_{\gamma_1}-\mu|_{\gamma_2}\|_W
    \,dm(\gamma)
    \\
    &\leq&
    2R(\delta)V(\mu)
    \\
    &\leq&
    2R(\delta)B_u.
\end{eqnarray*}
Thus,
\begin{equation}\label{ii}
    \int \func{II}(\gamma)\,dm(\gamma)
    \leq 2R(\delta)B_u.
\end{equation}
Summing \eqref{ia}, \eqref{ib} and \eqref{ii},
we conclude
\[
    \|(\func{F_0{_\ast}}-\func{F_\delta{_\ast}})
    \mu_\delta\|_1
    \leq
    \left(|G_0|_{\lip}+3B_u+2\right)R(\delta),
\]
as claimed.
\end{proof}

\subsection{Deterministic Perturbations of the Transfer Operator}

\begin{definition}\label{def:operator-family}
	Let \( (B_w,\|\cdot\|_w) \) and \( (B_s,\|\cdot\|_s) \)
be normed vector spaces of finite signed measures on
\( \Sigma \), with \( B_s\subseteq B_w \) and
\( \|\nu\|_w\leq\|\nu\|_s \) for every \( \nu\in B_s \).

Fix \( \delta_0\in(0,1) \). For each
\( \delta\in[0,\delta_0) \), let
\( \mathsf{T}_\delta:B_w\to B_w \) be a linear operator
such that \( \mathsf{T}_\delta(B_s)\subseteq B_s \).
Assume that there exists a probability measure
\( \mu_\delta\in B_s \) satisfying
\( \mathsf{T}_\delta\mu_\delta=\mu_\delta \).

Moreover, suppose that:
	\begin{enumerate}
		\item[(P1)] \textbf{Weak-norm perturbation bound.}
There exist \( C\in\mathbb{R}^+ \) and a function
\( R:[0,\delta_0)\to\mathbb{R}^+ \) such that 
		\[
		\lim_{\delta \to 0^+} R(\delta) = 0,
		\]
		and
		\[
		\|(\mathsf{T}_0 - \mathsf{T}_\delta)\mu_\delta\|_w \leq R(\delta) C,
		\quad \forall \delta \in [0,\delta_0).
		\]
		
		\item[(P2)] \textbf{Uniform strong-norm bound for fixed points.}
There exists \( M>0 \), independent of \( \delta \),
such that 
		\[
		\|\mu_\delta\|_s \leq M, 
		\quad \forall \delta \in [0,\delta_0).
		\]
		
		\item[(P3)] \textbf{Exponential convergence to equilibrium.}
The unperturbed operator \( \mathsf{T}_0 \) satisfies
the following strong-to-weak norm estimate on
zero-mass measures. Let
\[
    \mathcal{V}_s
    :=\{\mu\in B_s:\mu(\Sigma)=0\}.
\]
Then
\[
    \|\mathsf{T}_0^n\mu\|_w
    \leq C_2\rho_2^n\|\mu\|_s,
    \qquad
    \mu\in\mathcal{V}_s,\quad n\geq1.
\]
		
		\item[(P4)] \textbf{Uniform weak-norm bound for iterates.}
There exists \( M_2>0 \), independent of both
\( \delta \) and \( n \), such that 
		\[
		\|\mathsf{T}_\delta^n \nu\|_w \leq M_2 \, \|\nu\|_w,
		\quad \forall \delta \in [0,\delta_0), \ \forall n \in \mathbb{N}, \ \forall \nu \in B_s.
		\]
	\end{enumerate}
	
	A family of operators satisfying \textnormal{(P1)}, \textnormal{(P2)}, \textnormal{(P3)} and \textnormal{(P4)} is called an \textbf{$R(\delta)$-family of operators}.
\end{definition}

\begin{lemma}[Quantitative stability for fixed points
of operators]\label{dlogd}
Let \( \{\mathsf{T}_\delta\}_{\delta\in[0,\delta_0)} \)
be an \( R(\delta) \)-family of operators satisfying
\textnormal{(P1)--(P4)} of
Definition~\ref{def:operator-family}, with fixed
probability measures \( \mu_\delta\in B_s \).

Then there exist constants \( D_1>0 \) and
\( \delta_1\in(0,\delta_0) \) such that
\[
    \|\mu_\delta-\mu_0\|_w
    \leq
    D_1 R(\delta)|\log R(\delta)|,
    \qquad 0<\delta<\delta_1.
\]
\end{lemma}
To prove Lemma \ref{dlogd}, we state a general result on the stability of fixed points. 
For completeness, we include the elementary
telescoping argument underlying this estimate.

\begin{sublemma}\label{gen}
Let \( (\mathcal{B},\|\cdot\|_{\mathcal{B}}) \)
be a normed vector space, and let
\( \mathsf{T}_0,\mathsf{T}_\delta:
\mathcal{B}\to\mathcal{B} \) be linear operators.
Suppose that
\[
    \mathsf{T}_0f_0=f_0,
    \qquad
    \mathsf{T}_\delta f_\delta=f_\delta.
\]
Assume that, for every integer \( i\geq0 \),
there exists \( C_i>0 \), with \( C_0=1 \), such that
\[
    \|\mathsf{T}_0^i g\|_{\mathcal{B}}
    \leq C_i\|g\|_{\mathcal{B}},
    \qquad g\in\mathcal{B}.
\]
Then, for every integer \( N\geq1 \),
\[
\begin{split}
    \|f_\delta-f_0\|_{\mathcal{B}}
    \leq{}&
    \|\mathsf{T}_0^N(f_\delta-f_0)\|_{\mathcal{B}}
    \\
    &+
    \|(\mathsf{T}_\delta-\mathsf{T}_0)f_\delta
      \|_{\mathcal{B}}
    \sum_{i=0}^{N-1}C_i.
\end{split}
\]
\end{sublemma}

\begin{proof}
Since \( f_\delta=\mathsf{T}_\delta f_\delta \),
the telescoping identity gives
\[
\begin{split}
    f_\delta-\mathsf{T}_0^Nf_\delta
    &=
    \sum_{i=0}^{N-1}
    \mathsf{T}_0^i(f_\delta-\mathsf{T}_0f_\delta)
    \\
    &=
    \sum_{i=0}^{N-1}
    \mathsf{T}_0^i
    (\mathsf{T}_\delta-\mathsf{T}_0)f_\delta.
\end{split}
\]
Using \( \mathsf{T}_0^Nf_0=f_0 \), we obtain
\[
    f_\delta-f_0
    =
    \mathsf{T}_0^N(f_\delta-f_0)
    +
    \sum_{i=0}^{N-1}
    \mathsf{T}_0^i
    (\mathsf{T}_\delta-\mathsf{T}_0)f_\delta.
\]
The claimed inequality follows from the triangle
inequality and the bounds on \( \mathsf{T}_0^i \).
\end{proof}
\begin{proof}[Proof of Lemma~\ref{dlogd}]
Since \( \mu_\delta \) and \( \mu_0 \) are probability
measures in \( B_s \), their difference belongs to
the zero-mass subspace \( \mathcal{V}_s \).
By \textnormal{(P2)},
\[
    \|\mu_\delta-\mu_0\|_s\leq2M.
\]

Enlarging \( M_2 \) if necessary, assume that
\( M_2\geq1 \). Apply Sublemma~\ref{gen} to
\( \mathcal{B}=B_s \), equipped with the weak norm
\( \|\cdot\|_w \).
Condition \textnormal{(P4)}, applied to
\( \mathsf{T}_0 \), allows us to use
\( C_i\leq M_2 \) for every \( i\geq0 \).
Thus, by \textnormal{(P1)}, for every integer
\( N\geq1 \),
\[
    \|\mu_\delta-\mu_0\|_w
    \leq
    C M_2 N R(\delta)
    +
    \|\mathsf{T}_0^N(\mu_\delta-\mu_0)\|_w.
\]
Applying \textnormal{(P3)} and the strong-norm bound
above yields
\[
    \|\mu_\delta-\mu_0\|_w
    \leq
    C M_2 N R(\delta)
    +2C_2M\rho_2^N.
\]

The limit assumption in \textnormal{(P1)} implies
\( R(\delta)\to0 \) as \( \delta\to0^+ \).
Choose \( \delta_1\in(0,\delta_0) \) so that
\[
    0<R(\delta)<e^{-1},
    \qquad 0<\delta<\delta_1.
\]
For such \( \delta \), set
\[
    N=N(\delta)
    :=
    \left\lceil
        \frac{\log R(\delta)}{\log\rho_2}
    \right\rceil.
\]
Since \( 0<\rho_2<1 \), this choice gives
\[
    \rho_2^N\leq R(\delta),
    \qquad
    N\leq
    \frac{|\log R(\delta)|}{|\log\rho_2|}+1.
\]
Consequently,
\[
\begin{split}
    \|\mu_\delta-\mu_0\|_w
    \leq{}&
    \frac{CM_2}{|\log\rho_2|}
    R(\delta)|\log R(\delta)|
    \\
    &+
    (CM_2+2C_2M)R(\delta).
\end{split}
\]
Because \( |\log R(\delta)|\geq1 \), we conclude that
\[
    \|\mu_\delta-\mu_0\|_w
    \leq D_1R(\delta)|\log R(\delta)|,
\]
where
\[
    D_1
    :=
    CM_2\left(1+\frac{1}{|\log\rho_2|}\right)
    +2C_2M.
\]
\end{proof}

\subsection{Uniform Regularity Bounds and Quantitative Statistical Stability}

The next result asserts that the mapping
\[
\delta \longmapsto V(\mu_{\delta})
\]
is uniformly bounded, where $\{\mu_\delta\}_{\delta \in [0,1)}$ denotes the family of $F_\delta$-invariant probability measures associated with an admissible perturbation $\{F_\delta\}_{\delta \in [0,1)}$ of $F = F_0$. It is a key tool for proving Theorem~\ref{htyttigu} and other applications.

\begin{athm}\label{thshgf}
Let $\{F_\delta\}_{\delta\in[0,1)}$ be an admissible
$R(\delta)$-perturbation in the sense of
Definition~\ref{ad}, and let $\mu_\delta$ be the
unique $F_\delta$-invariant probability measure
in $S^1$, for each $\delta\in[0,\delta_0)$.
Then there exists a constant $B_u>0$, independent
of $\delta$, such that
\[
    V(\mu_\delta)\leq B_u,
\]
for all $\delta\in[0,\delta_0)$.
\end{athm}
We begin with the following preliminary lemma.
\begin{lemma}\label{las123rtryrdfd2}
If $\{F_\delta\}_{\delta\in[0,1)}$ is an admissible
$R(\delta)$-perturbation in the sense of
Definition~\ref{ad}, then there exist constants
$0<\beta_u<1$ and $D_{2,u}>0$, independent of
$\delta$, $n$, and $\mu$, such that, for every
positive measure $\mu\in\mathcal{BV}_m$ with
marginal density
\[
    \phi_1:=\frac{d\pi_{1*}\mu}{dm},
\]
we have
\begin{equation}\label{er}
    V\bigl(\overline F_{\delta*}^{\,n}\mu\bigr)
    \leq
    \beta_u^n V(\mu)
    +\frac{D_{2,u}}{1-\beta_u}\,|\phi_1|_v,
\end{equation}
for all $\delta\in[0,\delta_0)$ and all $n\geq0$.
Here $\overline F_\delta:=F_\delta^k$, where $k$
is the common iterate required by condition
\textnormal{(A2)} of Definition~\ref{ad}.
\end{lemma}
\begin{proof}
By condition \textnormal{(A2)} of
Definition~\ref{ad}, choose $\beta_u\in(0,1)$
such that
\[
    \sup_{\delta\in[0,\delta_0)}
    \alpha_{4,\delta}<\beta_u<1,
\]
and set
\[
    U_u:=
    \sup_{\delta\in[0,\delta_0)}U_{4,\delta}<\infty.
\]
Since $g_{\delta,k}$ is nonnegative and vanishes
at both endpoints of $I$, its pointwise variation
satisfies
\[
    2q_{k,\delta}
    \leq V_I(g_{\delta,k})
    \leq U_{4,\delta}.
\]
Consequently,
\[
    Q_u:=
    \sup_{\delta\in[0,\delta_0)}q_{k,\delta}
    \leq \frac{U_u}{2}<\infty.
\]

Fix $\delta\in[0,\delta_0)$ and a positive measure
$\mu\in\mathcal{BV}_m$. For $j\geq0$, write
\[
    \mu_j:=\overline F_{\delta*}^{\,j}\mu,
    \qquad
    \phi_j:=P_{f_\delta}^{kj}\phi_1,
\]
so that $\phi_j$ is the marginal density of
$\mu_j$ and $\phi_0=\phi_1$.
By Lemma~\ref{kjdfhkkhfdjfht}, we have
$\phi_1\in BV_m$, and repeated application of
that lemma shows that each $\mu_j$ is a positive
measure in $\mathcal{BV}_m$.

Let $D$ and $\lambda$ be the constants in
condition \textnormal{(A1)} and set
\[
    C_{\mathrm{base}}:=\max\{1,2D\}.
\]
For every $j\geq1$, condition \textnormal{(A1)}
gives
\[
\begin{aligned}
    |\phi_j|_v
    &\leq
    D\lambda^{kj}|\phi_1|_v
    +D\|\phi_1\|_{L^1(m)}\\
    &\leq 2D|\phi_1|_v.
\end{aligned}
\]
Including $j=0$, we therefore obtain
\begin{equation}\label{eq:uniform-marginal-bv}
    |\phi_j|_v
    \leq C_{\mathrm{base}}|\phi_1|_v,
    \qquad j\geq0.
\end{equation}

Applying Lemma~\ref{kjdfhkkhfdjfht} to $\mu_j$,
and using
$\|\mu_j\|_1=\|\phi_j\|_{L^1(m)}$, we obtain
\[
\begin{aligned}
    V(\mu_{j+1})
    &\leq
    \alpha_{4,\delta}V(\mu_j)
    +(q_{k,\delta}+U_{4,\delta})V(\phi_j)
    +U_{4,\delta}\|\phi_j\|_{L^1(m)}\\
    &\leq
    \beta_u V(\mu_j)
    +(Q_u+U_u)|\phi_j|_v\\
    &\leq
    \beta_u V(\mu_j)
    +D_{2,u}|\phi_1|_v,
\end{aligned}
\]
where
\[
    D_{2,u}:=(Q_u+U_u)C_{\mathrm{base}}.
\]
These constants are independent of $\delta$,
$j$, and $\mu$.

Iterating this inequality yields, for every $n\geq1$,
\[
\begin{aligned}
    V(\mu_n)
    &\leq
    \beta_u^nV(\mu)
    +D_{2,u}|\phi_1|_v
    \sum_{j=0}^{n-1}\beta_u^j\\
    &=
    \beta_u^nV(\mu)
    +D_{2,u}\frac{1-\beta_u^n}{1-\beta_u}
    |\phi_1|_v\\
    &\leq
    \beta_u^nV(\mu)
    +\frac{D_{2,u}}{1-\beta_u}|\phi_1|_v.
\end{aligned}
\]
The asserted inequality also holds for $n=0$,
which completes the proof.
\end{proof}

\begin{proof}[Proof of Theorem~\ref{thshgf}]
Let $\nu=m\times\nu_2$ be the probability measure
defined in Remark~\ref{riirorpdf}. Its marginal
density is identically equal to $1$, and
\[
    V(\nu)=0,
    \qquad |1|_v=1.
\]
Hence, by Lemma~\ref{las123rtryrdfd2},
\[
    V\bigl(\overline F_{\delta*}^{\,n}\nu\bigr)
    \leq \frac{D_{2,u}}{1-\beta_u}
\]
for every $\delta\in[0,\delta_0)$ and every $n\geq0$.
Set
\[
    B_u:=\frac{D_{2,u}}{1-\beta_u}.
\]

Fix $\delta\in[0,\delta_0)$ and write
\[
    \mu:=\mu_\delta,
    \qquad
    \mu_n:=\overline F_{\delta*}^{\,n}\nu,
    \qquad
    \overline F_\delta:=F_\delta^k.
\]
Since $\mu$ is $F_\delta$-invariant, it is also
$\overline F_\delta$-invariant. Moreover,
$\nu,\mu\in S^1$ are probability measures.
The spectral gap decomposition in
Theorem~\ref{spgap} therefore implies
\[
    \|\mu_n-\mu\|_1\longrightarrow0
    \qquad\text{as }n\to\infty.
\]

For each $n\geq1$, choose a path
$\Gamma_n\in\Gamma_{\mu_n}$, defined on a full
measure set $E_n\subset I$, such that
\[
    V_{E_n}(\Gamma_n)
    \leq V(\mu_n)+\frac1n
    \leq B_u+\frac1n.
\]
Such a choice is possible by the definition of
$V(\mu_n)$ as an infimum over representative paths.
Also fix a path $\Gamma\in\Gamma_\mu$, defined on
a full measure set $E_\mu\subset I$.

On the common full measure set
\[
    E:=E_\mu\cap\bigcap_{n\geq1}E_n,
\]
these paths satisfy
\[
    \int_E\|\Gamma_n(x)-\Gamma(x)\|_W\,dm(x)
    =\|\mu_n-\mu\|_1.
\]
Choose a strictly increasing sequence $(n_j)_{j\geq1}$
such that
\[
    \|\mu_{n_j}-\mu\|_1\leq2^{-j}.
\]
By Tonelli's theorem,
\[
\begin{aligned}
    \int_E
    \sum_{j\geq1}
    \|\Gamma_{n_j}(x)-\Gamma(x)\|_W\,dm(x)
    &=
    \sum_{j\geq1}\|\mu_{n_j}-\mu\|_1\\
    &<\infty.
\end{aligned}
\]
Consequently, there exists a full measure subset
$E'\subset E$ such that
\[
    \|\Gamma_{n_j}(x)-\Gamma(x)\|_W\longrightarrow0
    \qquad\text{for every }x\in E'.
\]

Let $x_0<\cdots<x_s$ be any finite ordered set
of points in $E'$. Pointwise convergence gives
\[
\begin{aligned}
    \sum_{\ell=1}^s
    \|\Gamma(x_\ell)-\Gamma(x_{\ell-1})\|_W
    &=
    \lim_{j\to\infty}
    \sum_{\ell=1}^s
    \|\Gamma_{n_j}(x_\ell)
      -\Gamma_{n_j}(x_{\ell-1})\|_W\\
    &\leq
    \lim_{j\to\infty}
    \left(B_u+\frac1{n_j}\right)
    =B_u.
\end{aligned}
\]
Taking the supremum over all such finite ordered
sets yields
\[
    V_{E'}(\Gamma)\leq B_u.
\]
Since $\Gamma|_{E'}$ represents $\mu=\mu_\delta$,
the definition of the variation of a measure gives
\[
    V(\mu_\delta)\leq V_{E'}(\Gamma)\leq B_u.
\]
The constant $B_u$ is independent of $\delta$.
Since $\delta\in[0,\delta_0)$ was arbitrary,
the proof is complete.
\end{proof}

\begin{athm}[Quantitative stability for deterministic
perturbations]\label{htyttigu}
Let \( \{F_\delta\}_{\delta\in[0,1)} \)
be an admissible \( R(\delta) \)-perturbation
in the sense of Definition~\ref{ad}.
For \( \delta\in[0,\delta_0) \), denote by
\( \mu_\delta\in S^1 \) its invariant probability
measure.

Then there exist constants \( D_2>0 \) and
\( \delta_1\in(0,\delta_0) \) such that
\begin{equation}\label{stabll}
    \|\mu_\delta-\mu_0\|_1
    \leq
    D_2R(\delta)|\log R(\delta)|,
    \qquad 0<\delta<\delta_1.
\end{equation}
In particular,
\( \|\mu_\delta-\mu_0\|_1\to0 \)
as \( \delta\to0^+ \).
\end{athm}

Before establishing Theorem~\ref{htyttigu}, we first prove Lemma~\ref{rrr} below.

\begin{lemma}\label{rrr}
Let \( \{F_\delta\}_{\delta\in[0,1)} \)
be an admissible \( R(\delta) \)-perturbation
in the sense of Definition~\ref{ad}.
Let \( \delta_0 \) be the parameter bound in that
definition, and let \( \mu_\delta\in S^1 \)
be the corresponding invariant probability measures
for \( \delta\in[0,\delta_0) \).

Then \( \{F_{\delta*}\}_{\delta\in[0,\delta_0)} \)
is an \( R(\delta) \)-family of operators in the
sense of Definition~\ref{def:operator-family},
with weak space \( (\mathcal{L}^1,\|\cdot\|_1) \)
and strong space \( (S^1,\|\cdot\|_{S^1}) \).
\end{lemma}

\begin{proof}
The pushforward operators are linear and preserve
\( \mathcal{L}^1 \).
Condition \textnormal{(A1)} and the weak contraction
from Proposition~\ref{weakcontral11234}
also imply that they preserve \( S^1 \).

We verify \textnormal{(P1)--(P4)}.

For \textnormal{(P1)}, Theorem~\ref{thshgf} provides
a constant \( B_u \), independent of \( \delta \),
such that
\[
    V(\mu_\delta)\leq B_u.
\]
Lemma~\ref{UF2ass} therefore gives
\[
    \|(F_{0*}-F_{\delta*})\mu_\delta\|_1
    \leq C_1R(\delta),
\]
with \( C_1 \) independent of \( \delta \).

To verify \textnormal{(P2)}, let
\[
    \phi_{1,\delta}
    :=
    \frac{d\pi_{1*}\mu_\delta}{dm}.
\]
Since \( \mu_\delta \) is an invariant probability
measure in \( S^1 \),
\[
    P_{f_\delta}\phi_{1,\delta}
    =\phi_{1,\delta},
    \qquad
    |\phi_{1,\delta}|_1=1,
    \qquad
    |\phi_{1,\delta}|_v<\infty.
\]
Applying \textnormal{(A1)} to this fixed density,
for every \( n\geq1 \), yields
\[
    |\phi_{1,\delta}|_v
    \leq
    D\lambda^n|\phi_{1,\delta}|_v+D.
\]
Letting \( n\to\infty \), we obtain
\[
    |\phi_{1,\delta}|_v\leq D.
\]
Moreover, positivity and normalization imply
\( \|\mu_\delta\|_1=1 \). Hence
\[
    \|\mu_\delta\|_{S^1}
    =
    |\phi_{1,\delta}|_v+\|\mu_\delta\|_1
    \leq D+1,
\]
which proves \textnormal{(P2)} with \( M=D+1 \).

Condition \textnormal{(P3)} follows from
Proposition~\ref{quasiquasiquasi} applied to \( F_0 \).

Finally, iterating the weak contraction from
Proposition~\ref{weakcontral11234} gives
\[
    \|F_{\delta*}^n\nu\|_1\leq\|\nu\|_1,
    \qquad
    \nu\in S^1,\quad n\geq0,
\]
uniformly for \( \delta\in[0,\delta_0) \).
Thus \textnormal{(P4)} holds with \( M_2=1 \).
\end{proof}
\begin{proof}[Proof of Theorem~\ref{htyttigu}]
By Lemma~\ref{rrr}, the family
\( \{F_{\delta*}\}_{\delta\in[0,\delta_0)} \)
satisfies \textnormal{(P1)--(P4)} with
\[
    B_w=\mathcal{L}^1,
    \qquad
    B_s=S^1.
\]
Applying Lemma~\ref{dlogd} to the invariant
probability measures \( \mu_\delta \)
gives \eqref{stabll}.

Since \( R(\delta)\to0 \) and
\( t|\log t|\to0 \) as \( t\to0^+ \),
the claimed statistical stability follows.
\end{proof}

When the perturbation bounds hold with
\[
    R(\delta)=K_6\delta
\]
for some constant $K_6>0$ and all sufficiently
small $\delta>0$, we obtain the following corollary.

\begin{corollary}[Quantitative stability for deterministic
perturbations with linear \( R(\delta) \)]
\label{htyttigui}
Let \( \{F_\delta\}_{\delta\in[0,1)} \)
be an admissible \( R(\delta) \)-perturbation
in the sense of Definition~\ref{ad}, with
\( R(\delta)=K_6\delta \) for some \( K_6>0 \).
For \( \delta\in[0,\delta_0) \), let
\( \mu_\delta\in S^1 \) be its invariant probability
measure.

Then there exist constants \( C_{6}>0 \)
and \( \delta_1\in(0,\delta_0) \) such that
\[
    \|\mu_\delta-\mu_0\|_1
    \leq
    C_{6}\delta|\log\delta|,
    \qquad 0<\delta<\delta_1.
\]
\end{corollary}

\begin{proof}
By Theorem~\ref{htyttigu}, for sufficiently small
\( \delta>0 \),
\[
    \|\mu_\delta-\mu_0\|_1
    \leq
    D_2K_6\delta|\log(K_6\delta)|.
\]
For \( 0<\delta<e^{-1} \),
\[
\begin{split}
    |\log(K_6\delta)|
    &\leq|\log\delta|+|\log K_6|
    \\
    &\leq
    (1+|\log K_6|)|\log\delta|.
\end{split}
\]
The result follows, after decreasing \( \delta_1 \)
if necessary, with
\[
    C_{6}
    :=
    D_2K_6(1+|\log K_6|)>0.
\]
\end{proof}

\section{Natural $C^r$ perturbations induce admissible $R(\delta)$-perturbations}\label{jhgfds}

In this section, we show that admissible $R(\delta)$-perturbations arise naturally from classical $C^r$ perturbations of the branches of the system. In particular, we provide explicit and verifiable conditions ensuring that such perturbations satisfy Definition~\ref{ad}. We also verify that the examples considered in this article fulfill these conditions.

\subsection{The models considered}

Let us assume that our dynamical system $F_0=(f_0,G_0)$ satisfies the following hypotheses.

\subsubsection{On $G_0$}\label{G_0}

We assume that \( G_0 \) satisfies the uniform
fiberwise contraction condition \textnormal{(H1)}
and the branchwise horizontal Lipschitz condition
\textnormal{(H2)}, both stated in Section~2.2. Moreover, we assume that $G_0$ admits a $C^2$ extension to the boundaries of the vertical strips $I_i \times [0,1]$ for all $i$.

\subsubsection{On $f_0$} \label{f_0}

We assume that the unperturbed map $f_0$ falls into one of the following two cases: 
\begin{enumerate}
    \item [(1)] $f_0$ satisfies Definition~\ref{def3}, exhibits strict convexity in each branch ($f_{0,i}'' \geq c_0 > 0$, for a constant $c_0$), and satisfies the summability condition $\sum_{i=1}^\infty \dfrac{1}{|f'_{0,i}(a_i)|} < + \infty$;
    \item [(2)] $f_0$ satisfies Definition~\ref{def2} and the summability condition $\sum_{i=1}^\infty \dfrac{1}{|f'_0(a_i)|} < + \infty$.
\end{enumerate}
Moreover, we assume that in both cases the map satisfies the following regularity and expansion conditions:
\begin{enumerate}
    \item [(a)] There exists $D > 0$ such that $\dfrac{|f_0''(x)|}{|f_0'(x)|^2} \le D$ for all $i \geq 1$ and $x \in \text{int}(I_i)$;
    \item [(b)] $|f_0'(x)| \geq 1$ for all $x$, and there exists $\lambda > 1$ such that $|(f_0^2)'(x)| \geq \lambda$. 
\end{enumerate}
The derivative-ratio bound in condition
\textnormal{(a)} and the expansion assumptions in
condition \textnormal{(b)} of this subsection yield
the following bounded distortion estimate for
some constant \( K>0 \):
\begin{equation}\label{adler}
    \frac{|(f_0^n)'(x)|}{|(f_0^n)'(y)|} \leq K
\end{equation}
for all $n \geq 1$ and $x,y$ in the same cylinder $J \in \mathcal{P}^{(n)}$.

Finally, we assume that $f_0(I_i) = (0,1)$ for all $i \ge 1$. In particular, $f_0$ is assumed to have continuous first and second order derivatives in the interior of $I_i$ for every $i$, which may blow up at the boundaries of the intervals.

\begin{remark}\label{oiytymkgf}
    Note that the Gauss map (Example~\ref{gauss}) satisfies case (1), as it is strictly convex, its derivatives are summable at the endpoints, and it satisfies Definition~\ref{def3} for its second iterate ($n_0=2$). On the other hand, the Lüroth map (Example~\ref{luroth}) directly satisfies case (2). Furthermore, in both models, the skew product
\( F_0 \) satisfies the iterate condition
\textnormal{(H3)} from Section~2.2 with \( k=1 \).
\end{remark}

\begin{remark}\label{exp}
    Note that, by Theorem \ref{raj}, if \( f_0 \) satisfies either case~(1) or case~(2)
of Section~5.1.2, there exists an iterate $N \in \mathbb{N}$ such that $\inf |(f_0^N)'| > 2$. 
\end{remark}

\subsection{On the perturbation}

\subsubsection{On $G_\delta$}\label{Gdelta}

We assume that each perturbed fiber map $G_\delta$ is a $C^2$ $\delta$-perturbation of $G_0$ when restricted to each vertical strip $I_i \times [0,1]$. Let $G_{\delta,i} = G_\delta|_{I_i \times [0,1]}$. We assume that for all $i \in \mathbb{N}$ and all $\delta \in [0, 1)$,
\[
\|G_{\delta,i} - G_{0,i}\|_{C^2} \leq \delta.
\]

\subsubsection{On $f_\delta$}\label{fdelta}

Let $\{f_\delta\}_\delta$ be a family of maps $f_\delta : \bigcup_{i=1}^\infty I_i \longrightarrow [0,1]$ such that:

\begin{enumerate}
    \item [(1)] If \( f_0 \) satisfies case~(1) of Section~5.1.2, we assume that $\|f_{0,i} - f_{\delta,i}\|_{C^2} \leq \delta$ for all $i$;
    \item [(2)] If \( f_0 \) satisfies case~(2) of Section~5.1.2, we assume that $\|f_{0,i} - f_{\delta,i}\|_{C^2} \leq \delta$ and for all $i$. 
\end{enumerate}
In particular, $f_\delta$ is $C^2$ in the interior of each interval $I_i = (a_i, b_i)$ for all $i$. In both cases, to simplify the calculations, we assume that $f_\delta|_{I_i} : I_i \longrightarrow (0,1)$ is a diffeomorphism for every $i$.

\begin{remark}\label{infdelta}
    By the derivative lower bound in condition
\textnormal{(b)} of Section~5.1.2 and the uniform
\( C^1 \)-proximity between \( f_0 \) and \( f_\delta \)
implied by the assumptions of Section~5.2.2, there exist constants $\tilde{\lambda} > 0$ and $\delta_1 \in (0,1)$ such that
    \begin{equation}
        \inf_{x} |f_\delta'(x)| \geq \tilde{\lambda}
    \end{equation}
    holds uniformly for all $\delta \in [0, \delta_1)$.
\end{remark}

Now we are in position to state the main theorem of this section.

\begin{athm}\label{tc}
    Suppose that $F_0 = (f_0, G_0)$ satisfies the hypotheses established in Sections \ref{G_0} and \ref{f_0}, and let $\{F_\delta\}_{\delta \in [0,1)}$ be a one-parameter family of systems satisfying the hypotheses of Sections \ref{Gdelta} and \ref{fdelta} for each $\delta$. Then, there exists $\delta_1 \in (0,1)$ such that the subfamily $\{F_\delta\}_{\delta \in [0,\delta_1)}$ is an $R(\delta)$-admissible perturbation in the sense of Definitions \ref{rdelta} and \ref{ad}.
\end{athm}

\subsection{Verification of condition (U1)}

The following lemmas verify condition
\textnormal{(U1)} of Definition~\ref{rdelta}
for small \( C^r \) perturbations of the base maps
described in cases~(1) and~(2) of Section~5.1.2.

\begin{lemma}
Let \( f_0 \) be a map satisfying case~(1)
of Section~5.1.2; that is, $f_0$ satisfies Definition~\ref{def3}, exhibits strict convexity uniformly bounded away from zero (there exists $c_0 > 0$ such that $f_{0,i}''(x) \geq c_0$ for all $i \in \mathbb{N}$ and $x \in I_i$), and satisfies the summability condition $\sum_{i=1}^\infty |f'_{0,i}(a_i)|^{-1} < \infty$. If $f_\delta$ satisfies the perturbation hypotheses
of Section~\ref{fdelta}, then, for \( \delta \) sufficiently small, \( f_\delta \) satisfies condition \textnormal{(U1)} of Definition~\ref{rdelta}. Moreover, each $f_\delta$ is strictly convex.
\end{lemma}

\begin{proof}
Since $\sup_{i \in \mathbb{N}} \|f_{0,i} - f_{\delta,i}\|_{C^2} < \delta$, we have for all $x \in I_i$:
\[
f_{\delta,i}''(x) \geq f_{0,i}''(x) - \delta \geq c_0 - \delta.
\]
Choosing $\delta < c_0$ ensures that $f_{\delta,i}''(x) > 0$ for all $i \in \mathbb{N}$, thereby preserving strict convexity. 

Furthermore, the $C^1$-proximity guarantees that for small $\delta$, the derivatives at the left endpoints $f_{\delta,i}'(a_i)$ remain close to $f_{0,i}'(a_i)$, preserving the summability condition. Indeed, for each $i \in \mathbb{N}$, we have:
\[
|f_{\delta,i}'(a_i)| \geq |f_{0,i}'(a_i)| - \delta.
\]
Because $f_0$ satisfies the summability condition $\sum_{i \geq 1} |f_{0,i}'(a_i)|^{-1} < \infty$, it follows that $|f_{0,i}'(a_i)| \to \infty$ as $i \to \infty$. Thus, there exists an integer $N \in \mathbb{N}$ such that $|f_{0,i}'(a_i)| > 2\delta$ for all $i > N$. For such $i$, we obtain:
\[
|f_{\delta,i}'(a_i)| > |f_{0,i}'(a_i)| - \frac{1}{2}|f_{0,i}'(a_i)| = \frac{1}{2}|f_{0,i}'(a_i)|.
\]
Consequently, the tail of the perturbed series satisfies:
\[
\sum_{i=N+1}^{\infty} \frac{1}{|f_{\delta,i}'(a_i)|} < 2 \sum_{i=N+1}^{\infty} \frac{1}{|f_{0,i}'(a_i)|} < \infty.
\]
The finite sum for $i=1, \dots, N$ is strictly positive and well-defined provided we choose $\delta < \min_{1 \le i \le N} |f_{0,i}'(a_i)|$. Thus, $\sum_{i=1}^{\infty} |f_{\delta,i}'(a_i)|^{-1} < \infty$. We next show that $f_\delta^2$ is uniformly
expanding. For each branch of $f_0$, define
\[
    \ell_i(t):=
    \log\left|
        f_{0,i}'\bigl(f_{0,i}^{-1}(t)\bigr)
    \right|,
    \qquad t\in(0,1).
\]
By condition \textnormal{(a)} of Section~5.1.2,
\[
    |\ell_i'(t)|\leq D.
\]
For $x\in I_i$, set
\[
    z:=f_{0,i}^{-1}(f_{\delta,i}(x)).
\]
This is well-defined because both branches map
$I_i$ onto $(0,1)$. Since
\[
    |f_{0,i}(x)-f_{0,i}(z)|
    =
    |f_{0,i}(x)-f_{\delta,i}(x)|
    \leq\delta,
\]
we obtain
\[
    |f_0'(x)|\geq e^{-D\delta}|f_0'(z)|.
\]
Moreover, the $C^1$-proximity and the bound
$|f_0'|\geq1$ imply, for $\delta<1$,
\[
    |f_\delta'(u)|
    \geq |f_0'(u)|-\delta
    \geq(1-\delta)|f_0'(u)|.
\]
Consequently, at every point where the
corresponding iterates are defined,
\[
\begin{aligned}
    |(f_\delta^2)'(x)|
    &\geq
    (1-\delta)^2
    |f_0'(x)|\,|f_0'(f_\delta(x))|\\
    &\geq
    (1-\delta)^2e^{-D\delta}
    |f_0'(z)|\,|f_0'(f_0(z))|\\
    &\geq
    (1-\delta)^2e^{-D\delta}\lambda.
\end{aligned}
\]
Since $\lambda>1$, the last expression is greater
than $1$ for all sufficiently small $\delta$.
Thus $f_\delta^2$ is uniformly expanding
on its natural branch partition.

The derivative-ratio bound also persists:
\[
    \frac{|f_\delta''(x)|}{|f_\delta'(x)|^2}
    \leq
    \frac{|f_0''(x)|+\delta}
         {(1-\delta)^2|f_0'(x)|^2}
    \leq
    \frac{D+\delta}{(1-\delta)^2}.
\]
Together with eventual uniform expansion,
this yields bounded distortion on all cylinders
by the standard inverse-branch estimates. Let $J$ be an open branch interval of
$f_\delta^2$, and write
\[
    w_J:=\frac{1}{|(f_\delta^2)'|}\bigg|_J.
\]
The chain rule and the preceding bounds give
a constant $C>0$, independent of $J$ and
of sufficiently small $\delta$, such that
\[
    |w_J'|\leq C.
\]
Moreover, since $f_\delta^2(J)=(0,1)$,
bounded distortion gives
\[
    \sup_J w_J\leq K\,m(J),
\]
where $K$ is a uniform distortion constant.
Consequently, the reciprocal derivative
$g_{\delta,2}$, extended by zero outside
the branch interiors, satisfies
\[
    V_I(g_{\delta,2})
    \leq
    \sum_J\left(Cm(J)+2Km(J)\right)
    =C+2K<\infty.
\]
Therefore, $f_\delta^2$ satisfies
Definition~\ref{def2}. 
The corresponding bounded-variation arguments,
as developed in \cite[Section~3]{RG2}, show that
$P_{f_\delta}$ is bounded on $BV_m$ and that
the Cesaro averages of $P_{f_\delta}^n1$ are
bounded in $BV_m$. Compactness in $L^1(m)$
therefore provides an invariant probability
density $h_\delta\in BV_m$.

Since every branch is onto $(0,1)$, the
distortion bound gives upper and positive lower
bounds for $P_{f_\delta}^n1$, independent of $n$,
and hence for the density $h_\delta$ obtained
above. The full-branch property, bounded
distortion, and shrinking cylinders then give,
by the standard density-point argument,
\[
    m_\delta(f_\delta^n A)\longrightarrow1
\]
for every measurable set $A$ with
$m_\delta(A)>0$, where $m_\delta:=h_\delta m$.
Thus $f_\delta$ is exact with respect to
$m_\delta$, and in particular $m_\delta$ is mixing.
Since $m_\delta$ is equivalent to $m$,
ergodicity also implies that $m_\delta$ is
the unique invariant probability measure
absolutely continuous with respect to $m$.

Finally, $|f_\delta'|\geq1-\delta>0$ implies
nonsingularity. We have therefore verified all
the requirements of Definition~\ref{def3}.
Since $f_\delta(I_i)=f_0(I_i)=(0,1)$ for every
$i$, condition \textnormal{(U1)} follows.
\end{proof}

The next lemma establishes the analogous result for piecewise expanding maps.

\begin{lemma}
Let \( f_0 \) be a map satisfying case~(2)
of Section~5.1.2; that is, $f_0$ satisfies Definition~\ref{def2} with uniform expansion constant $\lambda_0 > 1$ ($\inf_{i \in \mathbb{N}} \inf_{x \in I_i} |f_{0,i}'(x)| \geq \lambda_0$) and the summability condition $\sum_{i=1}^\infty |f'_{0,i}(a_i)|^{-1} < \infty$. If $f_\delta$ is a piecewise $C^1$ map defined on the same partition as $f_0$, and $\sup_{i \in \mathbb{N}} \|f_{0,i} - f_{\delta,i}\|_{C^1} < \delta$, then, for \( \delta \) sufficiently small,
\( f_\delta \) satisfies condition
\textnormal{(U1)} of Definition~\ref{rdelta}.
\end{lemma}

\begin{proof}
By the uniform $C^1$-proximity, for each branch $i \in \mathbb{N}$ and any $x \in I_i$, we have:
\[
|f_{\delta,i}'(x)| \geq |f_{0,i}'(x)| - \delta \geq \lambda_0 - \delta.
\]
By choosing $\delta < \lambda_0 - 1$, we define $\lambda_\delta = \lambda_0 - \delta > 1$. Therefore, $|f_{\delta,i}'(x)| \geq \lambda_\delta > 1$ for all $x$, meaning $f_\delta$ satisfies the uniform expansion requirement.

Let $g_\delta=1/|f_\delta'|$ on the branch
interiors, extended by zero elsewhere.
The uniform derivative-ratio bound controls
its variation inside the branches.
Moreover, endpoint summability is preserved
by the $C^1$-proximity, and bounded distortion
controls the right endpoint values by the
left endpoint values. Thus,
\[
    V_I(g_\delta)<\infty.
\]
Together with uniform expansion, this verifies
Definition~\ref{def2} for $f_\delta$.

The same full-branch and bounded-distortion
arguments used in the preceding proof give
boundedness of $P_{f_\delta}$ on $BV_m$
and a unique absolutely continuous invariant
probability measure with a $BV_m$ density,
which is mixing.
Hence $f_\delta$ satisfies
Definition~\ref{def3} with $n_0=1$.
Since $f_\delta(I_i)=f_0(I_i)=(0,1)$
for every $i$, condition \textnormal{(U1)} follows.

\end{proof}

\subsection{Verification of condition (U2)}

\subsubsection{Condition (U2.3)}

We first verify the uniform closeness of the fiber maps
required by condition \textnormal{(U2.3)}
of Definition~\ref{rdelta}. To ensure this property holds, it suffices to assume that the restriction of $G_0$ to each vertical strip $I_i \times K$ admits a $C^2$ extension to the boundary, and that each perturbed fiber map $G_\delta$ is a $C^2$ $\delta$-perturbation of $G_0$ when restricted to these strips. That is, for all $i \in \mathbb{N}$ and all $\delta \in [0, \delta_0)$, we have
\[
\|G_\delta - G_0\|_{C^2(I_i \times K)} \leq \delta.
\]
It is worth noting that this $C^2$ proximity elegantly controls the necessary dynamical constants: the supremum of the partial derivatives in the $y$-direction bounds the contraction rate along the fibers, while the supremum of the partial derivatives in the $x$-direction yields the uniform horizontal Lipschitz constant. Consequently, the assumed \( C^2 \)-proximity yields
the uniform closeness of the fiber maps required by
\textnormal{(U2.3)} of Definition~\ref{rdelta},
with \( R(\delta)=\mathcal{O}(\delta) \).

\subsubsection{Condition (U2.2)}

By the branchwise \( C^2 \)-perturbation assumptions
of Section~5.2.2, for $\delta$ sufficiently small, $f_\delta$ inherits the robust properties of $f_0$. Specifically, $f_\delta$ preserves the expansion condition (with $|(f_\delta^2)'(x)| \geq \tilde{\lambda} > 1$) and the uniform distortion bound: there exists $\tilde{D} > 0$ such that 

\begin{equation}\label{adlerd}
    \sup_{x \in I_i} |f_\delta''(x)|/|f_\delta'(x)|^2 \leq \tilde{D}
\end{equation}for all $i \geq 1$. As a direct consequence of the robustness of these two properties, $f_\delta$ automatically satisfies the global bounded distortion condition for all iterates. That is, there exists a uniform constant $K > 0$ such that for all $n \geq 1$, all cylinders $J\in\mathcal{P}_\delta^{(n)}$,
where $\mathcal{P}_\delta^{(n)}$ denotes the
natural branch partition of $f_\delta^n$, and all $x,y \in J$, we have 
\begin{equation}
    |(f_\delta^n)'(x)| \leq K |(f_\delta^n)'(y)|.
\end{equation}In particular, for $n=1$, we obtain

\begin{equation}\label{ituryt}
|f_\delta'(x)| \leq K |f_\delta'(y)|    
\end{equation} for all $x,y \in I_i$, which is the crucial property for the following estimates.

The following lemma establishes the inverse-branch
proximity condition \textnormal{(U2.2)} of
Definition~\ref{rdelta} and provides a uniform
Lipschitz bound for the reciprocal derivatives
on each branch.

\begin{lemma}\label{lem:geom_consequences}
Assume that \( f_0 \) satisfies the hypotheses
of either case~(1) or case~(2) of Section~5.1.2,
and let \( f_\delta \) be a \( \delta \)-perturbation
satisfying the corresponding assumptions
of Section~5.2.2. Then, for $\delta$ sufficiently small, the following properties hold:
\begin{enumerate}
    \item[(i)] There exists a constant $C_1 > 0$ such that, for all $i \geq 1$ and all $\gamma \in I$,
    \[
    |\gamma_{\delta,i} - \gamma_{0,i}| \leq \frac{C_1}{|f'_\delta(a_i)|} \delta.
    \]
    In particular,
\[
    \operatorname*{ess\,sup}_{\gamma\in(0,1)}
    \sup_{i\geq1}
    |\gamma_{0,i}-\gamma_{\delta,i}|
    \leq 2K\delta,
\]
which verifies the inverse-branch proximity
condition \textnormal{(U2.2)} of
Definition~\ref{rdelta}.
    
    \item[(ii)] There exists a constant $C_2 > 0$ such that for all $i \geq 1$ and all $x,y \in I_i$,
    \[
    \left| \frac{1}{f_\delta'(x)} - \frac{1}{f_\delta'(y)} \right| \leq C_2 |x-y|.
    \]
\end{enumerate}
\end{lemma}

\begin{proof}

Fix $\gamma \in I$ and let $\gamma_{0,i}, \gamma_{\delta,i} \in I_i$ satisfy $f_0(\gamma_{0,i}) = \gamma$ and $f_\delta(\gamma_{\delta,i}) = \gamma$. Then
\[
|f_\delta(\gamma_{\delta,i}) - f_\delta(\gamma_{0,i})| = |\gamma - f_\delta(\gamma_{0,i})| = |f_0(\gamma_{0,i}) - f_\delta(\gamma_{0,i})| \leq \|f_\delta - f_0\|_\infty \leq \delta.
\]
By the Mean Value Theorem and the estimate above there exists $\xi \in I_i$ such that
\[
|f_\delta(\gamma_{\delta,i}) - f_\delta(\gamma_{0,i})| = |f_\delta'(\xi)| \cdot |\gamma_{\delta,i} - \gamma_{0,i}| \leq \delta \implies |\gamma_{\delta,i} - \gamma_{0,i}| \leq \frac{\delta}{|f_\delta'(\xi)|}.
\]
Here, $f_\delta'(a_i)$ denotes the right-hand
limit of the derivative of the $i$-th branch.
Taking the limit as $x\to a_i^+$ in the
branchwise bounded distortion estimate
\eqref{ituryt}, we obtain
\[
    |f_\delta'(a_i)|
    \leq K|f_\delta'(\xi)|.
\] Hence,
\[
|\gamma_{\delta,i} - \gamma_{0,i}| \leq \frac{K \delta}{|f'_\delta(a_i)|}.
\]
This proves (i) with $C_1 = K$. Since \( |f_0'(x)|\geq1 \) by condition
\textnormal{(b)} of Section~5.1.2, for $\delta$ sufficiently small we are guaranteed that $|f_\delta'(a_i)| \geq 1/2$. Thus, the distance between corresponding inverse images
is uniformly bounded by \( 2K\delta \), verifying the
inverse-branch proximity condition \textnormal{(U2.2)}
of Definition~\ref{rdelta}, with
\( R(\delta)=\mathcal{O}(\delta) \).

By the Mean Value Theorem, for any $x,y \in I_i$, there exists $\eta \in I_i$ such that
\[
\left| \frac{1}{f_\delta'(x)} - \frac{1}{f_\delta'(y)} \right| = \left| \left(\frac{1}{f_\delta'}\right)'(\eta) \right| \cdot |x-y| = \frac{|f_\delta''(\eta)|}{|f_\delta'(\eta)|^2} \cdot |x-y|.
\]Applying the uniform derivative-ratio bound
\eqref{adlerd}, we obtain, with
\( C_2:=\tilde{D}>0 \),
\[
\left| \frac{1}{f_\delta'(x)} - \frac{1}{f_\delta'(y)} \right| \leq C_2 |x-y|.
\]
This proves (ii).
\end{proof}

\subsubsection{Condition U2.1}

The following lemma establishes the
reciprocal-derivative perturbation bound
\textnormal{(U2.1)} of Definition~\ref{rdelta}.

\begin{lemma}\label{lem:U21}
Assume the hypotheses of Lemma~\ref{lem:geom_consequences}. Then there exists a constant $C>0$ such that, for all $\delta$ sufficiently small,
\[
\sum_{i=1}^{\infty} \left| \frac{1}{f_\delta'(\gamma_{\delta,i})} - \frac{1}{f_0'(\gamma_{0,i})} \right| \leq C \delta,
\]
which verifies the reciprocal-derivative perturbation
bound \textnormal{(U2.1)} of Definition~\ref{rdelta},
with \( R(\delta)=\mathcal{O}(\delta) \).
\end{lemma}

\begin{proof}
Fix $\gamma \in I$. For each $i \geq 1$, the triangle inequality yields:
\[
\left| \frac{1}{f_\delta'(\gamma_{\delta,i})} - \frac{1}{f_0'(\gamma_{0,i})} \right| \leq \underbrace{\left| \frac{1}{f_\delta'(\gamma_{\delta,i})} - \frac{1}{f_\delta'(\gamma_{0,i})} \right|}_{\text{I}} + \underbrace{\left| \frac{1}{f_\delta'(\gamma_{0,i})} - \frac{1}{f_0'(\gamma_{0,i})} \right|}_{\text{II}}.
\]

\noindent
\textbf{Estimate of I:}
Combining the Lipschitz bound for reciprocal
derivatives from Lemma~\ref{lem:geom_consequences}(ii)
with the inverse-branch proximity estimate from
Lemma~\ref{lem:geom_consequences}(i), we obtain
\[
\text{I} \leq C_2 |\gamma_{\delta,i} - \gamma_{0,i}| \leq C_2 \frac{C_1}{|f'_\delta(a_i)|} \delta.
\]
Set
\[
    S:=\sum_{i=1}^{\infty}
    \frac{1}{|f_0'(a_i)|}<\infty.
\]
Since $|f_0'(a_i)|\geq1$ and the branchwise
$C^1$-proximity also holds at the one-sided
endpoint limits, for $0\leq\delta<1/2$ we have
\[
    |f_\delta'(a_i)|
    \geq |f_0'(a_i)|-\delta
    \geq (1-\delta)|f_0'(a_i)|
    \geq \frac12|f_0'(a_i)|.
\]
Consequently,
\[
    \sum_{i=1}^{\infty}
    \frac{1}{|f_\delta'(a_i)|}
    \leq 2S,
\]
uniformly in $\delta$. Therefore,
\[
    \sum_{i=1}^{\infty}\text{I}
    \leq 2C_1C_2S\delta.
\]
\medskip
\noindent
\textbf{Estimate of II:}
Using the identity $|a^{-1} - b^{-1}| = |b - a| / |ab|$, we obtain:
\[
\text{II} = \frac{\left| f_0'(\gamma_{0,i}) - f_\delta'(\gamma_{0,i}) \right|}{|f_\delta'(\gamma_{0,i})| \, |f_0'(\gamma_{0,i})|}.
\]
The branchwise \( C^2 \)-proximity assumed in
Section~5.2.2 implies
\( \|f'_{\delta,i}-f'_{0,i}\|_\infty\leq\delta \)
for every branch \( i \). Hence, the numerator
is bounded by \( \delta \). For the denominator, we use the bounded distortion
estimate \eqref{adler} for \( f_0 \) with \( n=1 \)
and its perturbed counterpart \eqref{ituryt}
for \( f_\delta \), which give $|f_0'(\gamma_{0,i})| \geq K^{-1} |f_0'(a_i)|$ and $|f_\delta'(\gamma_{0,i})| \geq K^{-1} |f_\delta'(a_i)|$. Furthermore, for $\delta$  sufficiently small, $|f_\delta'(a_i)| \geq \frac{1}{2} |f_0'(a_i)|$. Therefore,
\[
|f_\delta'(\gamma_{0,i})| \, |f_0'(\gamma_{0,i})| \geq \frac{1}{2 K^2} |f_0'(a_i)|^2.
\]
Substituting this back, we find:
\[
\text{II} \leq \frac{2 K^2 \delta}{|f_0'(a_i)|^2} \leq \frac{2 K^2 \delta}{|f_0'(a_i)|},
\]
where the last inequality uses the derivative
lower bound \( |f_0'(a_i)|\geq1 \) from condition
\textnormal{(b)} of Section~5.1.2. Summing over $i$, the convergence of $\sum_{i=1}^\infty |f_0'(a_i)|^{-1}$ ensures that $\sum_{i=1}^\infty \text{II} = \mathcal{O}(\delta)$.

\medskip
\noindent

Combining the estimates for terms I and II,
we obtain the reciprocal-derivative perturbation
bound \textnormal{(U2.1)} of Definition~\ref{rdelta},
with the sum bounded by \( C\delta \)
for some constant \( C>0 \).
\end{proof}

\subsection{Obtaining condition (A1)}

We recall the results and constructions of
Rychlik~\cite{RLK} used below to establish the
uniform Lasota--Yorke inequality required by
condition \textnormal{(A1)} of Definition~\ref{ad}.

\subsubsection{Results on Rychlik's article}

Let $X$ be a totally ordered order-complete set. Open intervals constitute a base of a compact topology in $X$. We fix a regular, Borel probabilistic measure $m$ on $X$.

\subsubsection*{The Dynamical System and Rychlik's Assumptions}

Let $T: U \to X$ be a continuous map, where $U \subset X$ is open and dense, and $m(U) = 1$. Let $S = X \setminus U$ denote the set of singularities. We assume the following conditions:

\begin{itemize}
    \item There exists a countable family $\beta$ of closed intervals with disjoint interiors such that $\bigcup _{B \in \beta} B \supset U$ and, for any $B \in \beta$, the set $B \cap S$ consists exactly of the endpoints of $B$.
    \item For any $B \in \beta$, the restriction $T|_{B \cap U}$ admits an extension to a homeomorphism of $B$ with some interval in $X$.
    \item A weight function $g: X \to \mathbb{R}_+$ is given, satisfying $\|g\|_\infty < 1$, $V g < +\infty$, and $g|_S = 0$.
\end{itemize}

The Perron-Frobenius operator $\operatorname{P}: L_m^1 \to L_m^1$ associated with $T$ and $g$ is defined by:
\begin{equation}
    \operatorname{P}f(x) = \sum_{y \in T^{-1}(x)} g(y)f(y).
\end{equation}
Rychlik assumes that $P$ preserves the measure $m$, which means $m(Pf) = m(f)$ for every $f \in L_1$, and proves that $J = 1/g$ acts as the Jacobian of $T$.

To deal with the iterations of $T$, we define for $N \geq 1$: $S_N = \bigcup_{k=0}^{N-1} T^{-k}(S)$, $U_N = X \setminus S_N$, and the refined partition $\beta^N = \bigvee_{k=0}^{N-1} T^{-k}(\beta)$. The weight function for the $N$-th iterate is given by $$g_N|_{S_N} = 0$$ and $$g_N|_{U_N} = g \cdot (g \circ T) \cdot \ldots \cdot (g \circ T^{N-1}).$$

The following decomposition of the variation over the branch partition is Lemma~1 of Rychlik~\cite{RLK}. In the following identity, $u$ is represented by a function of bounded pointwise variation,
and all variations are pointwise, including the endpoint values of $ug$, which vanishes on $S$.

\begin{lemma}[Rychlik's Lemma 1] \label{lem:rychlik_1}
    For every $u \in BV_m$,
    \[
        V(u \cdot g) = \sum_{B \in \beta} V_B(u \cdot g).
    \]
\end{lemma}

\begin{proposition}\label{uniformLY}
Under the hypotheses of Theorem~\ref{tc},
there exist $\delta_1>0$, an integer $N\geq1$,
and constants $\lambda_0\in(0,1)$ and $D>0$,
independent of $\delta$, such that
\[
    V(\operatorname{P}_{f_\delta}^{N}u)
    \leq \lambda_0 V(u)+D\|u\|_1
\]
for every $u\in BV_m$ and
$\delta\in[0,\delta_1)$.
Here, $V$ denotes the variation of an
$L^1_m$ equivalence class.
\end{proposition}

\begin{proof}
By the uniform expansion estimates established
above, after decreasing $\delta_1$ if necessary,
there exist constants $s\in(0,1]$ and
$\Lambda>1$ such that
\[
    |f_\delta'|\geq s,
    \qquad
    |(f_\delta^2)'|\geq\Lambda
\]
on the corresponding branch interiors,
uniformly for $\delta\in[0,\delta_1)$.
Consequently, writing
\[
    q_{\delta,n}
    :=\operatorname*{ess\,sup}
       \frac{1}{|(f_\delta^n)'|},
    \qquad q_{\delta,0}:=1,
\]
we have
\[
    q_{\delta,n}
    \leq s^{-1}\Lambda^{-\lfloor n/2\rfloor},
    \qquad n\geq0.
\]

Let $A>0$ be a uniform bound in
\eqref{adlerd}. The chain rule gives, on
each branch of $f_\delta^N$,
\[
    \frac{|(f_\delta^N)''(x)|}
         {|(f_\delta^N)'(x)|^2}
    \leq
    A\sum_{j=0}^{N-1}
    \frac{1}{
      |(f_\delta^{N-j-1})'(f_\delta^{j+1}(x))|
    }.
\]
Thus,
\[
    \frac{|(f_\delta^N)''(x)|}
         {|(f_\delta^N)'(x)|^2}
    \leq
    A\sum_{\ell=0}^{N-1}q_{\delta,\ell}
    \leq
    \frac{2A}{s(1-\Lambda^{-1})}
    =:D,
\]
uniformly in both $\delta$ and $N$.

Fix $N$ and let $\mathcal{J}_{\delta,N}$
be the family of open branch intervals of
$f_\delta^N$. Since every branch of $f_\delta$
maps onto $(0,1)$, the same holds for every
$J\in\mathcal{J}_{\delta,N}$.
Denote the corresponding inverse branch by
$v_J:(0,1)\to J$, and set
\[
    w_J(x):=\frac{1}{|(f_\delta^N)'(x)|},
    \qquad x\in J.
\]
Then
\[
    \operatorname{P}_{f_\delta}^N u
    =
    \sum_{J\in\mathcal{J}_{\delta,N}}
    (u w_J)\circ v_J
    \quad\text{almost everywhere}.
\]

Choose a representative of $u$ with minimal
pointwise variation. Since $v_J$ is monotone,
the product estimate for a BV function and
a $C^1$ weight yields
\[
\begin{aligned}
    V_{(0,1)}\big((u w_J)\circ v_J\big)
    &=V_J(u w_J)\\
    &\leq
    q_{\delta,N}V_J(u)
    +\int_J |u(x)|\,|w_J'(x)|\,dx\\
    &\leq
    q_{\delta,N}V_J(u)
    +D\int_J |u(x)|\,dx.
\end{aligned}
\]
These are variations on open intervals;
no extension by zero is used in this estimate.

The branch contributions converge absolutely
in $L^1$, since
\[
    \sum_{J\in\mathcal{J}_{\delta,N}}
    \|(u w_J)\circ v_J\|_1
    =
    \sum_{J\in\mathcal{J}_{\delta,N}}
    \int_J |u|\,dm
    =
    \|u\|_1.
\]
Summing the preceding variation estimates
and using lower semicontinuity of variation
under $L^1$ convergence, we obtain
\[
    V(\operatorname{P}_{f_\delta}^N u)
    \leq q_{\delta,N}V(u)+D\|u\|_1.
\]
Finally, choose a single $N$ sufficiently large
that
\[
    \lambda_0
    :=s^{-1}\Lambda^{-\lfloor N/2\rfloor}<1.
\]
This proves the assertion uniformly for
$\delta\in[0,\delta_1)$.

\end{proof}
\subsection{Obtaining condition (A2)}

We now verify the uniform bounds for the variation
estimates required by the admissibility condition
\textnormal{(A2)} of Definition~\ref{ad}.
The auxiliary results below will be used to control
the constants \( \alpha_{4,\delta} \) and
\( U_{4,\delta} \) appearing in that condition.

\subsubsection{Some additional results}
\begin{proposition}\label{kjdfghdfg}
Let $X = [0,1]$ be partitioned as $X = \left(\bigcup_{i=1}^\infty I_i\right) \cup S$, where $I_i = (a_i, b_i)$ are disjoint open intervals and $S$ is the set of boundary points. 
Let $g: X \to \mathbb{R}$ be a function of bounded variation such that $g|_S = 0$. Assume that the lateral limits $g(a_i^+)$ and $g(b_i^-)$ exist for all $i$. Then, the pointwise variation of this representative of $g$ on $X$ is given by:
$$
V_{X}(g) = \sum_{i=1}^\infty \Big( V_{I_i}(g) + |g(a_i^+)| + |g(b_i^-)| \Big).
$$
\end{proposition}

\begin{proof}
Define a family of functions $g_i : X \to \mathbb{R}$, where each $g_i|_{I_i}=g|_{I_i}$, that is:
$$
g_i(x) = 
\begin{cases} 
g(x), & \text{if } x \in I_i \\ 
0, & \text{if } x \notin I_i.
\end{cases}
$$
Since $g(x) = 0$ for all $x \in S$, and the sets $I_i$ are mutually disjoint, it holds that
$$
g(x) = \sum_{i=1}^\infty g_i(x) \quad \text{for all } x \in X.
$$Let us compute the variation of $g_i$ over the entire space $X$. 
Consider a finite sequence $a_i < x_1 < x_2 < \dots < x_n < b_i$. The variation sum along this sequence is
$$
|g_i(x_1) - g_i(a_i)| + \sum_{k=1}^{n-1} |g_i(x_{k+1}) - g_i(x_k)| + |g_i(b_i) - g_i(x_n)|.
$$
Since $g_i(a_i) = 0$ and $g_i(b_i) = 0$, this reduces to:
$$
|g(x_1)| + \sum_{k=1}^{n-1} |g(x_{k+1}) - g(x_k)| + |g(x_n)|.
$$
Taking the supremum over all such finite sequences inside $I_i$, the middle term yields the variation $V_{I_i}(g)$, while the extreme points $x_1$ and $x_n$ can be chosen arbitrarily close to $a_i$ and $b_i$. Thus:
$$
V_{X}(g_i) = |g(a_i^+)| +V_{I_i}(g) + |g(b_i^-)|.
$$For each $n\geq1$, set
\[
    h_n:=\sum_{i=1}^n g_i.
\]
Since these functions are supported on disjoint
intervals and vanish at their endpoints,
finite additivity of pointwise variation gives
\[
    V_X(h_n)=\sum_{i=1}^n V_X(g_i).
\]
Moreover,
\[
    \sum_{i=1}^n V_X(g_i)
    =
    \sum_{i=1}^n V_{[a_i,b_i]}(g)
    \leq V_X(g).
\]
Letting $n\to\infty$, we obtain
\[
    \sum_{i=1}^{\infty}V_X(g_i)\leq V_X(g).
\]
Conversely, $h_n\to g$ pointwise on $X$.
Lower semicontinuity of pointwise variation
therefore yields
\[
    V_X(g)
    \leq \liminf_{n\to\infty}V_X(h_n)
    =\sum_{i=1}^{\infty}V_X(g_i).
\]
Combining both inequalities and the formula
for $V_X(g_i)$, we conclude that
\[
    V_X(g)
    =
    \sum_{i=1}^{\infty}
    \Big(
        V_{I_i}(g)
        +|g(a_i^+)|
        +|g(b_i^-)|
    \Big).
\]
\end{proof}

\begin{proposition}\label{mcvnc}
Let $f_0, f_\delta: U \to X$ be maps satisfying the hypothesis of Theorem \ref{tc}, where $X=[0,1]$ and $U = \bigcup B_i^\circ$ for a partition $\beta = \{B_i\}$ of closed intervals $B_i = [a_i, b_i]$. 
Define the functions $g_0 = 1/f_0'$ and $g_\delta = 1/f_\delta'$ on $U$, with $g_0|_S = g_\delta|_S = 0$ for $S = X \setminus U$. Then there exist constants $P > 0$ and $\delta_1 \in (0,1)$ such that:
$$
V_X(g_0 - g_\delta) \leq P \delta,
$$for all $\delta \in [0,\delta _1)$.
\end{proposition}

\begin{proof}

Let $K_{\mathrm{dist}}$ be a common branchwise
distortion constant for $f_0$ and $f_\delta$.
By the endpoint summability assumption,
\[
    S_0:=\sum_i |g_0(a_i^+)|<\infty.
\]
Bounded distortion gives
\[
    \sum_i |g_0(b_i^-)|
    \leq K_{\mathrm{dist}}S_0.
\]
Set
\[
    S(g_0):=
    \max\left\{
        \sum_i |g_0(a_i^+)|,\,
        \sum_i |g_0(b_i^-)|
    \right\}<\infty.
\]
Moreover, Adler's condition implies
\[
    \sum_i V_{I_i}(g_0)
    =
    \sum_i\int_{I_i}|g_0'|\,dm
    \leq D.
\]
The endpoint comparison
$|f_\delta'(a_i^+)|\geq
\frac12|f_0'(a_i^+)|$, valid for sufficiently
small $\delta$, together with uniform bounded
distortion and \eqref{adlerd}, gives the
analogous finite bounds for $g_\delta$.
Consequently, the zero extensions $g_0$ and
$g_\delta$ have bounded pointwise variation,
as follows by taking finite sums of their
branch restrictions and passing to the
pointwise limit.

By the uniform derivative lower bound from
Remark~\ref{infdelta}, there exist
\( \tilde{\lambda}>0 \) and \( \delta_1\in(0,1) \)
such that
\( |f_0'|,|f_\delta'|\geq\tilde{\lambda} \)
for all \( \delta\in[0,\delta_1) \). Define $h = g_0 - g_\delta$. Since $g_0|_S = 0$ and $g_\delta|_S = 0$, we have $h|_S = 0$. Applying the variation decomposition from
Proposition~\ref{kjdfghdfg} to \( h=g_0-g_\delta \),
we separate the contributions from the branch
interiors and the jumps at their endpoints:
$$
V_X(h) = \sum_{B_i \in \beta} \left( |h(a_i^+)| + \int_{a_i}^{b_i} |h'(x)| \, dx + |h(b_i^-)| \right).
$$ Combining the derivative proximity implied by
the \( C^2 \)-perturbation assumptions of
Section~5.2.2 with the uniform derivative lower
bound from Remark~\ref{infdelta}, we have
\( |f_\delta'-f_0'|\leq\delta \) and
\( |f_\delta'|\geq\tilde{\lambda} \).
Therefore,
$$
|h(a_i^+)| = \frac{|f_\delta'(a_i^+) - f_0'(a_i^+)|}{|f_0'(a_i^+) f_\delta'(a_i^+)|} \leq \frac{\delta}{\tilde{\lambda} |f_0'(a_i^+)|} = \frac{\delta}{\tilde{\lambda}} |g_0(a_i^+)|.
$$
The same argument at the right endpoints gives
\[
    |h(b_i^-)|
    \leq
    \frac{\delta}{\tilde{\lambda}}
    |g_0(b_i^-)|.
\]
Hence,
\[
    \sum_i
    \bigl(|h(a_i^+)|+|h(b_i^-)|\bigr)
    \leq
    \frac{2S(g_0)}{\tilde{\lambda}}\delta.
\]

Computing the derivative $h'(x)$ in the interior of the intervals and noting that $\frac{|f_0''|}{(f_0')^2} = |g_0'(x)|$:
$$
|h'(x)| \leq \frac{|f_\delta''(x) - f_0''(x)|}{(f_\delta'(x))^2} + |g_0'(x)| \frac{|f_0'(x) - f_\delta'(x)| \cdot |f_0'(x) + f_\delta'(x)|}{(f_\delta'(x))^2}.
$$
Using the \( C^2 \)-perturbation bounds from
Section~5.2.2, the uniform derivative lower bound
from Remark~\ref{infdelta}, and Adler's condition
\textnormal{(a)} from Section~5.1.2, which gives
\( |g_0'(x)|\leq D \), we obtain
$$
|h'(x)| \leq \frac{\delta}{\tilde{\lambda}^2} + D \delta \left( \frac{|f_0'(x)|}{(f_\delta'(x))^2} + \frac{1}{|f_\delta'(x)|} \right).
$$
Since $|f_0'(x)|\leq |f_\delta'(x)|+\delta$,
for $\delta\in[0,\delta_1)$ we have
\[
    \frac{|f_0'(x)|}{|f_\delta'(x)|^2}
    +\frac{1}{|f_\delta'(x)|}
    \leq
    \frac{2}{\tilde{\lambda}}
    +\frac{\delta_1}{\tilde{\lambda}^2}
    =:K.
\] Thus, the derivative is uniformly bounded by:
$$
|h'(x)| \leq \delta \left( \frac{1}{\tilde{\lambda}^2} + D K \right).
$$
Integrating over $I_i$, we get:
$$
\int_{I_i} |h'(x)| \, dx \leq \delta \left( \frac{1}{\tilde{\lambda}^2} + D K \right) m(I_i).
$$Combining the endpoint estimates with the
integral estimate on each branch, and using
\( \sum_i m(I_i)=1 \), we obtain
$$
V_X(h) \leq \delta \left[ \frac{2S(g_0)}{\tilde{\lambda}} + \left( \frac{1}{\tilde{\lambda}^2} + D K \right) \sum_i m(I_i) \right] = \delta \left[ \frac{2S(g_0)}{\tilde{\lambda}} + \frac{1}{\tilde{\lambda}^2} + D K \right].
$$
Taking the term in brackets as the constant $P$, we conclude that $V_X(g_0 - g_\delta) \leq P \delta$.
\end{proof}

\begin{proposition}\label{prop:A2}
Let $\{F_\delta\}_{\delta \in [0,\delta_0)}$ be a family of systems satisfying the hypotheses of Theorem \ref{tc}. Then, for \( \delta \) sufficiently small, the family
satisfies the uniform bounds required by the
admissibility assumption \textnormal{(A2)}
of Definition~\ref{ad}.
\end{proposition}

\begin{proof}
We verify \textnormal{(A2)} with $k=1$.
By condition \textnormal{(b)} of Section~5.1.2
and the uniform $C^1$-proximity,
\[
    |f_\delta'|\geq1-\delta.
\]
Moreover, the fiber contraction rate and the
branchwise horizontal Lipschitz constant satisfy
\[
    \alpha_\delta\leq\alpha_0+\delta,
    \qquad
    |G_\delta|_{\operatorname{Lip}}
    \leq |G_0|_{\operatorname{Lip}}+\delta.
\]
Choose $\delta_1\in(0,\delta_0)$ sufficiently
small that
\[
    \delta_1<\frac{1-\alpha_0}{2}.
\]
Then
\[
    \sup_{\delta\in[0,\delta_1)}
    \alpha_{4,\delta}
    \leq
    \frac{\alpha_0+\delta_1}{1-\delta_1}
    <1.
\]
Together with the horizontal Lipschitz bound,
this also verifies \textnormal{(H3)} with
the common iterate $k=1$.
To bound the quantity \( U_{4,\delta} \) appearing
in the admissibility assumption \textnormal{(A2)}
of Definition~\ref{ad}, recall that, for \( k=1 \),
$$
U_{4,\delta} := |G_\delta|_{\operatorname{Lip}} \cdot \operatorname{ess\,sup} \frac{1}{|f_\delta'|} + V_{I} \left( \frac{1}{|f_\delta'|} \right).
$$

By the uniform $C^2$-proximity of $G_\delta$ to $G_0$, the Lipschitz constant $|G_\delta|_{\operatorname{Lip}}$ is uniformly bounded for small $\delta$. Furthermore, let $g_0 = 1/f_0'$ and $g_\delta = 1/f_\delta'$. By the variation estimate for the difference
of reciprocal derivatives in
Proposition~\ref{mcvnc}, we have
\( V_I(g_0-g_\delta)\leq P\delta \)
for all \( \delta\in[0,\delta_1) \), after
decreasing $\delta_1$ if necessary.

Here, $g_0$ and $g_\delta$ are extended by zero
on the common singular set, as in
Proposition~\ref{mcvnc}. Since taking absolute
values does not increase pointwise variation,
\[
\begin{aligned}
    V_I\left(\frac{1}{|f_\delta'|}\right)
    &=V_I(|g_\delta|)\\
    &\leq V_I(g_\delta)\\
    &\leq V_I(g_0)+V_I(g_0-g_\delta)\\
    &\leq V_I(g_0)+P\delta.
\end{aligned}
\]
The finiteness of $V_I(g_0)$ was established
in the proof of Proposition~\ref{mcvnc},
using Adler's condition, endpoint summability,
and bounded distortion. Consequently,
\[
    \sup_{\delta\in[0,\delta_1)}
    V_I\left(\frac{1}{|f_\delta'|}\right)
    \leq V_I(g_0)+P\delta_1<\infty.
\]
Combining these bounds, we obtain
\( \sup_{\delta\in[0,\delta_1)} U_{4,\delta}<\infty \).
Together with the previously established bound
\( \sup_{\delta\in[0,\delta_1)}\alpha_{4,\delta}<1 \),
this verifies the admissibility assumption
\textnormal{(A2)} of Definition~\ref{ad}
and completes the proof.
\end{proof}

\end{document}